%% file: main_arxiv_v2.tex
\documentclass[11pt,reqno]{amsart}
\usepackage[margin=1in]{geometry}

\input{preamble.tex}
\usepackage[justification=centering]{caption}

\title[Sparse OU process with i.i.d.~paths]{Sparse estimation for the drift of high-dimensional Ornstein--Uhlenbeck processes with i.i.d.~paths}
\author[S Nakakita]{Shogo Nakakita}
\address{Komaba Institute for Science, University of Tokyo, 3-8-1 Komaba, Meguro-ku, Tokyo 153-8902, Japan}
\date{}

\usepackage[T1]{fontenc}
\usepackage{newtxtext,newtxmath}

\begin{document}
\begin{abstract}
    We study sparsity-regularized maximum likelihood estimation for the drift parameter of high-dimensional non-stationary Ornstein--Uhlenbeck processes given repeated measurements of i.i.d.~paths.
    In particular, we show that Lasso and Slope estimators can achieve the minimax optimal rate of convergence.
    We exhibit numerical experiments for sparse estimation methods and show their performance.
\end{abstract}

\maketitle

\section{Introduction}

Recent increases in available data motivate us to employ more information in analysing dynamical phenomena with high-dimensional statistical models.
To meet this expectation, statistical theories and methods for continuous-time high-dimensional dynamical systems have been developed in this decade.
Although theories and methods of high-dimensional statistical models have gathered great interest in the community for a long time \citep{candes2007dantzig,van2008high,bartlett2020benign}, estimation of high-dimensional dynamical systems is a quite recent topic.
Especially, estimation of high-dimensional discrete-time stochastic dynamical systems has been explosively investigated in this decade \citep{basu2015regularized,wong2020lasso,nakakita2026benign}.
Moreover, statistics for high-dimensional continuous-time stochastic dynamical models is a further emerging topic developed in these years \citep{fujimori2019dantzig,gaiffas2019sparse,ciolek2020dantzig,ciolek2025lasso,dexheimer2024lasso,marushkevych2026consistent}.
They usually assume the ergodicity of the solutions of stochastic differential equations (SDEs) and long-term observations, and studies under setups without ergodicity are very scarce, though such situations frequently appear in real data analysis.

This study analyses drift estimation of a high-dimensional Ornstein--Uhlenbeck (OU) process, which is a fundamental statistical model for continuous-time stochastic dynamics, with repeated measurements of independent and identically distributed (i.i.d.) paths instead of ergodicity and long-term observations.
The development of estimation methods for non-ergodic OU processes enables us to consider statistical modelling of non-ergodic high-dimensional stochastic dynamics, which often appear in analysis of longitudinal data tracking individuals for a certain period.
For example, pharmacokinetic (PK) and pharmacodynamic (PD) data are often longitudinal data of independent individuals, and thus statistical modelling with stochastic differential equations with multiple independent path observations is one of the promising approaches \citep{donnet2013review}; one of the motivating examples is the analysis of drug concentration dynamics in the body.
Besides, this type of SDE modelling is also applied to the study of aging, health, lifespan, and mortality by tracking multiple individuals \citep{yashin2007stochastic,yashin2012quadratic}.
In this study, we investigate estimation of OU processes with sparse drift coefficients by observing i.i.d.~paths from a theoretical point of view.
While studies under the i.i.d.~observation scheme have focused on low-dimensional parametric or nonparametric inference \citep{delattre2013maximum,comte2020nonparametric,marie2023nadaraya}, we show that high-dimensional inference can be successfully discussed under this framework.

Our main contribution is to show that the rates of convergence of Lasso and Slope estimators, which are typical sparse estimation methods recently, are minimax optimal under the repeated measurement regime.
In particular, the rates of convergence of them under our regime are the same as those in previous studies \citep[e.g.,][]{dexheimer2024lasso} except for the sample size $N$ in our setup and the terminal $T$ in the previous works.
Therefore, our result affirmatively answers the question ``Does Lasso (or Slope) for OU processes in repeated measurements have favourable theoretical guarantees?'', which should naturally arise given the recent development of high-dimensional estimation for ergodic SDEs with long-term observations.
Moreover, the result also indicates the possibility that other results derived under ergodicity (for example, sparse estimation without linearity as \citealp{ciolek2025lasso}) may hold even under the repeated measurement scheme.

In addition, as a technical contribution, we develop a novel concentration bound for a sample covariance matrix with the presence of non-centred structures.
While some previous studies have used the centred property of stochastic processes to evaluate the concentration of sample covariance matrices, we need to extend their argument due to the non-centred evolution of OU processes.
Specifically, we decompose a sample covariance matrix into its second-order Wiener chaos part, which is common in the analysis of sparse estimation for centred processes, its first-order Wiener chaos part, which appears due to the non-centred behaviour, and the part of the sample covariance matrix of transformed initial values.
As a result, the derived concentration bounds of all of them are of the same order, and thus, we can derive the result in a similar manner to the previous studies.

\subsection{Literature reviews}
OU processes are fundamental continuous-time stochastic processes, and they are frequently used as statistical models of dynamical systems.
They have been initiated as a physical model \citep{uhlenbeck1930theory}, and applications have been greatly diversified: for example, finance \citep{vasicek1977equilibrium}, neuroscience \citep{ricciardi1979ornstein}, biology \citep{bachar2012stochastic}, and wind power estimation \citep{arenas2020ornstein}.

Non-asymptotic analysis of sparse estimation of SDEs with ergodicity has gathered interest.
\citet{fujimori2019dantzig} considers theoretical guarantees for the Dantzig selector of linearly parametrized diffusion processes with ergodicity.
\citet{gaiffas2019sparse} give non-asymptotic bounds for $\ell^{1}$-regularized maximum likelihood estimation of ergodic Ornstein--Uhlenbeck processes with row-sparsity.
\citet{ciolek2020dantzig} study the Lasso estimator and Dantzig selector for the model, similar to \citet{gaiffas2019sparse}, and improve bounds.
\citet{ciolek2025lasso} consider a generalized model setup and give non-asymptotic bounds without the linearity of models.
\citet{dexheimer2024lasso} investigate Lasso and Slope estimation for L\'{e}vy-driven Ornstein--Uhlenbeck processes.

Estimation of the drift coefficient of SDEs with independent and identically distributed (i.i.d.) paths has been studied in these two decades, as we can find this problem in the context of functional data analysis and longitudinal/panel data analysis.
We can classify the approaches of previous studies into parametric and nonparametric estimation.
In particular, the parametric estimation problem has been strongly tied to mixed-effect models.
\citet{ditlevsen2005mixed} study SDEs with mixed effects and estimation of them. 
\citet{picchini2011practical} consider computationally efficient estimation of high-dimensional SDEs with mixed effects.
\citet{delattre2013maximum} investigate maximum likelihood estimation of SDEs with random effects.
The nonparametric estimation is another important direction in estimation via i.i.d.~path observations.
\citet{comte2020nonparametric} investigate nonparametric estimation of drift coefficients with a projection estimator defined in $L^{2}$.
\citet{marie2023nadaraya} analyse Nadaraya--Watson estimation of drift coefficients.
\citet{comte2024nonparametric} consider nonparametric estimation of the drift of SDEs with space-time dependent drift and diffusion coefficients.
\citet{ella2024nonparametric} studies nonparametric estimation of the diffusion coefficient of SDEs with i.i.d.~path observations.
Recently, \citet{bayrakter2026sparse} investigated the Lasso estimation problem for SDEs with drift functions with nonlinearity in state variables via discretely observed i.i.d.~paths, whose result is the nearest neighbourhood of this study.
This work considers both of Lasso and Slope as well as minimax lower bound matching upper bounds on them; therefore, the focus of this study does not completely overlap with that of \citet{bayrakter2026sparse}.

\subsection{Notation}
We introduce a list of the notation used in this study.

For any $d_{1}\times d_{2}$ matrix $\bfM$, $\bfM^{(i,j)}$ denotes the $(i,j)$-th element of $\bfM$, $\bfM^{\top}$ is the transpose of $\bfM$, and $\bfM^{\otimes2}:=\bfM\bfM^{\top}$.
For any $d\times d$ matrix $\bfV$, $\tr(\bfV)=\sum_{i=1}^{d}\bfV^{(i,i)}$, that is the trace of $\bfV$. 
For any pair of $d_{1}\times d_{2}$ matrices $\bfM_{1},\bfM_{2}$, $\langle \bfM_{1},\bfM_{2}\rangle =\tr(\bfM_{1}^{\top}\bfM_{2})$.
For any $d_{1}\times d_{2}$ matrix $\bfM_{1}$ and $d_{3}\times d_{4}$ matrix $\bfM_{2}$, $\bfM_{1}\otimes\bfM_{2}$ denotes the Kronecker product defined as
\begin{equation*}
    \bfM_{1}\otimes\bfM_{2}=\left[\begin{matrix}
        \bfM_{1}^{(1,1)}\bfM_{2} & \cdots & \bfM_{1}^{(1,d_{2})}\bfM_{2} \\
        \vdots & \ddots & \vdots\\
        \bfM_{1}^{(d_{1},1)}\bfM_{2} & \cdots & \bfM_{1}^{(d_{1},d_{2})}\bfM_{2}
    \end{matrix}\right]\in\R^{(d_{1}d_{3})\times(d_{2}d_{4})}.
\end{equation*}

For any $d_{1}\times d_{2}$ matrix $\bfM$, we define the $\ell^{p}$-norm ($p\in\{0\}\cup[1,\infty]$) as 
\begin{equation*}
    \|\bfM\|_{p}:=\begin{cases}
        \sum_{i=1}^{d_{1}}\sum_{j=1}^{d_{2}}\ind_{(0,\infty)}(|\bfM^{(i,j)}|)&\text{ if }p=0,\\
        \left(\sum_{i=1}^{d_{1}}\sum_{j=1}^{d_{2}}|\bfM^{(i,j)}|^{p}\right)^{1/p}&\text{ if }p\in[1,\infty),\\
        \max_{i,j}|\bfM^{(i,j)}|&\text{ if }p=\infty.
    \end{cases}
\end{equation*}
For the same $\bfM$, we define the operator norm of $\bfM$ as 
\begin{equation*}
    \|\bfM\|_{\op}:=\sup_{\bu\in\bbS^{d_{1}-1},\bv\in\bbS^{d_{2}-1}}\bu^{\top}\bfM\bv.
\end{equation*}

For a vector $\bv\in\R^{d_{0}}$, let $\bv_{1}^{\sharp},\ldots,\bv_{d_{0}}^{\sharp}$ denote a nonincreasing rearrangement of $|\bv_{1}|,\ldots,|\bv_{d_{0}}|$, and we define $\|\bv\|_{\ast}$ as
\begin{equation*}
    \|\bv\|_{\ast}=\sum_{i=1}^{d_{0}}\lambda_{i}\bv_{i}^{\sharp},
\end{equation*}
where $\{\lambda_{i}\colon i=1,\ldots,d_{0}\}$ is a sequence of non-increasing positive numbers.
Throughout this paper, we let $\lambda_{i}=\sqrt{\log(2d_{0}/i)}$ as \citet{dexheimer2024lasso}.
For $d_{1}\times d_{2}$ matrix $\bfM$, we define $\|\bfM\|_{\ast}=\|\textnormal{vec}(\bfM)\|_{\ast}$.

For any pair $\bx,\by\in L^{2}:=L^{2}([0,T],\diff t;\R^{d})$, we define the $L^{2}$ inner product
\begin{equation*}
    \langle \bx,\by\rangle_{L^{2}}:=\int_{0}^{T}\bx(t)^{\top}\by(t)\diff t
\end{equation*}
and $\|\cdot\|_{L^{2}}$ is the $L^{2}$-norm induced by this inner product.

For a real-valued random variable $X$, $\|X\|_{\psi_{2}}$, the subgaussian norm of $X$, is defined as 
\begin{equation*}
    \|X\|_{\psi_{2}}:=\inf\left\{t>0\colon \E\left[\exp\left(X^{2}/t^{2}\right)\right]\le 2\right\}.
\end{equation*}

\subsection{Paper organization}
Section \ref{sec:problem} explains the problem setup of this paper in detail.
Section \ref{sec:theoretical} is devoted to the presentation of the theoretical analysis on Lasso and Slope estimation for high-dimensional OU processes with repeated measurement.
Section \ref{sec:numerical} shows the results of a numerical experiment and confirms that practical behaviours of the sparse estimation coincide with the theoretical analysis.
Section \ref{sec:conclusion} is for concluding remarks and future studies.
All technical proofs are deferred to the Appendix.
\section{Problem setup}\label{sec:problem}

We first define the following $d$-dimensional stochastic differential equations (SDEs):
\begin{equation*}
    \diff\bx_{i}(t)=\bfA\bx_{i}(t)\diff t+\diff \bw_{i}(t),\ \bx_{i}(0)=\bxi_{i},\ t\in[0,T],\ i\in\{1,...,N\},
\end{equation*}
where $\bfA\in\R^{d\times d}$ is the unknown drift parameter, 
$\bw=\{\bw_{i}(t)\colon t\in[0,T],i\in\{1,...,N\}\}$ is a sequence of independent $d$-dimensional standard Wiener processes, $\{\bxi_{i}\}_{i=1}^{N}$ is a sequence of $d$-dimensional random vectors independent of the Wiener processes, and $T>0$ is the terminal.
Our problem is estimation of the $d^{2}$-dimensional parameter $\bfA$ given the continuous observations of the $N$ solutions $\{\bx_{i}(t)\colon t\in[0,T],i\in\{1,...,N\}\}$.
We let $\bfA_{0}$ denote the true value of $\bfA$.

We define an empirical risk function as follows:
\begin{equation}\label{eq:risk}
    \Loss_{N}(\bfA):=\frac{1}{N}\sum_{i=1}^{N}\left(-\int_{0}^{T}\left(\bfA \bx_{i}(t)\right)^{\top}\diff\bx_{i}(t)+\frac{1}{2}\int_{0}^{T}\left\|\bfA\bx_{i}(t)\right\|_{2}^{2}\diff t\right),\ \bfA\in\R^{d\times d}.
\end{equation}
This is the negative log-likelihood function scaled by $1/N$.
If dimension $d$ is fixed, then the asymptotic efficiency of its minimizer (the maximum likelihood estimator, MLE) is a well-known result from classical statistical theory and studies in these decades.
However, since our interest is estimation under large $d$, the validity of MLE can be violated.
Therefore, instead of the minimization of the empirical risk, we consider the following regularized empirical risk minimization problem:
\begin{equation}\label{eq:erm}
    \minimize_{\bfA\in\R^{d\times d}}\left\{\Loss_{N}(\bfA)+\lambda\Reg(\bfA)\right\},
\end{equation}
where $\Reg:\R^{d\times d}\to[0,\infty)$ is the regularization function and $\lambda$ is the magnitude of regularization.
Successful selection of the regularizer $\Reg(\cdot)$ enables us to introduce an inductive bias in high-dimensional estimation problems and makes the problems tractable in a similar manner to low-dimensional problems.
The Lasso regularization \citep{tibshirani1996regression} and Slope regularization \citep{bogdan2015slope} are typical to introduce sparsity as an inductive bias.
When choosing $\Reg(\cdot)=\|\cdot\|_{1}$, we derive the Lasso estimator $\hat{\bfA}_{\Lasso}$:
\begin{equation}\label{eq:lasso:definition}
    \hat{\bfA}_{\Lasso}\in\argmin_{\bfA\in\R^{d\times d}}\left\{\Loss_{N}(\bfA)+\lambda_{\Lasso}\|\bfA\|_{1}\right\}.
\end{equation}

Similarly, by taking $\Reg(\cdot)=\|\cdot\|_{\ast}$, the Slope estimator $\hat{\bfA}_{\Slope}$ is defined as
\begin{equation}\label{eq:slope:definition}
    \hat{\bfA}_{\Slope}\in\argmin_{\bfA\in\R^{d\times d}}\left\{\Loss_{N}(\bfA)+\lambda_{\Slope}\|\bfA\|_{\ast}\right\}.
\end{equation}
In the following sections, we investigate the theoretical and practical behaviours of these regularized estimators under $N$, which can be much smaller than the dimension of unknown parameters $d^{2}$.

Let us remark on the asymptotic regime where our non-asymptotic bounds given below get non-vacuous.
Our discussion is based on an implicit assumption that both $s\ll N$ ($s=\|\bfA_{0}\|_{0}$) and $d\ll N$ hold.
Note that the parameter dimension is $d^{2}$ (not $d$), and usually $d=\calO( s)$ (in many situations, each coordinate should depend on itself, and thus most of the diagonal elements of $\bfA_{0}$ should be non-zero).
A more detailed discussion is given in Section \ref{sec:theoretical}.

\section{Theoretical results}\label{sec:theoretical}
In this section, we exhibit upper bounds on errors of regularized MLEs under mild conditions.
We also see that lower bounds on minimax risks, and upper and lower bounds match up to constants; therefore, Lasso estimation is a minimax rate-optimal estimation in our i.i.d.~observation framework as well as long-term observations shown in \citet{ciolek2020dantzig}.

\subsection{Assumptions}\label{sec:theoretical:assumptions}

We impose a subgaussianity assumption on the initial values.
\begin{assumption}\label{assm:initial}
    $\{\bxi_{i}\colon i\in\{1,\ldots,N\}\}$ is a sequence of independent and identically distributed subgaussian random vectors; that is, for some $K\ge 1$, for any $x\in\R^{d}$, $\|\langle\bxi_{i},x\rangle\|_{\psi_{2}}\le K\|\langle\bxi_{i},x\rangle\|_{L^{2}}$.
\end{assumption}
Let $\bSigma=\E[\bxi_{i}\bxi_{i}^{\top}]$.
This assumption is used for the derivation of high-probability bounds for the operator norm of a sample second moment matrix $\|(1/N)\sum_{i=1}^{N}\bxi_{i}\bxi_{i}^{\top}\|_{\op}$.
As long as we can obtain an upper bound of $\calO(1)$, we can replace this assumption with different ones.

We define the following second moment matrix:
\begin{align*}
    \bfC_{\infty}&:=\int_{0}^{T}\left(\exp(t\bfA_{0})\bSigma\exp(t\bfA_{0})^{\top}+\int_{0}^{t}\exp((t-t')\bfA_{0})^{\otimes2}\diff t'\right)\diff t\\
    &=\E\left[\frac{1}{N}\sum_{i=1}^{N}\int_{0}^{T}\bx_{i}(t)\bx_{i}(t)^{\top}\diff t\right].
\end{align*}
We also let $\kappa_{\max}$ and $\kappa_{\min}$ as the largest eigenvalue and smallest one of $\bfC_{\infty}$.
Note the bounds $\kappa_{\max}\le  (\|\bSigma\|_{\op}+T)T\exp\left(2T\|\bfA_{0}\|_{\op}\right)$ and $\kappa_{\min}\ge (1/2)T^{2}\exp\left(-2T\|\bfA_{0}\|_{\op}\right)$ as $\|\bu\|_{2}=\|e^{-s\bfA_{0}^\top}e^{s\bfA_{0}^\top}\bu\|_{2}\le e^{s\|\bfA_{0}\|_{\op}}\|e^{s\bfA_{0}^\top}\bu\|_{2}$ for any $s\ge0$ and $\bu\in\R^{d}$
\citep[e.g.,][]{hall2015lie}.
We define
\begin{equation*}
    \kappa^\ast=\kappa_{\max}+\frac{\kappa_{\min}}{2}.
\end{equation*}

\subsection{Upper bounds for errors}

The next proposition derives an upper bound on the error of the Lasso estimator.
\begin{proposition}\label{prop:upper:lasso}
    Suppose that Assumption \ref{assm:initial} holds, and let $s=\|\bfA_{0}\|_{0}$.
    There exists a universal constant $c_{\Lasso}>0$ such that if $\lambda_{\Lasso}$ satisfies
    \begin{equation*}
        \lambda_{\Lasso}\ge 2c_{\Lasso}\sqrt{\frac{\kappa^{\ast}\log(2ed^{2}/s)}{N}},
    \end{equation*}
    then, for some $c\ge 1$ dependent only on $T$, $\|\bfA_0\|_{\op}$, $K$, and $\|\bSigma\|_{\op}$, for any $\epsilon_{0}\in(0,1)$ and $N\in\N$, with probability at least $1-\epsilon_{0}/2-8\times9^{d}\times e^{-N/(c(1+d/N))}$, for all $\bfA\in\R^{d\times d}$ with $\|\bfA\|_{0}\le s$ simultaneously,
    \begin{align*}
        \frac{1}{N}\sum_{i=1}^{N}\left\|\left(\hat{\bfA}_{\Lasso}-\bfA_{0}\right)\bx_{i}\right\|_{L^{2}}^{2}+\lambda_{\Lasso}\left\|\bfA-\hat{\bfA}_{\Lasso}\right\|_{1}&\le \frac{1}{N}\sum_{i=1}^{N}\left\|\left(\bfA-\bfA_{0}\right)\bx_{i}\right\|_{L^{2}}^{2}\\
        &\quad+\frac{8\lambda_{\Lasso}^{2}}{\kappa_{\min}}\left(s\vee \frac{\log(4\epsilon_{0}^{-1})}{\log(2ed^{2}/s)}\right).
    \end{align*}
\end{proposition}

We can apply this result to derive the bounds for the errors of Lasso under various norms.
Let us define the following positive integer: for each $\epsilon_{1}>0$,
\begin{equation*}
    N_{0}(\epsilon_{1}):=\min\left\{N\in\N\colon 8\times9^{d}\times \exp\left(-\frac{N}{c(1+d/N)}\right)\le \frac{\epsilon_{1}}{2}\right\}.
\end{equation*}
For the sample size $N>N_{0}(\epsilon_{1})$, we obtain the following high-probability bounds for the errors of the Lasso estimation.
\begin{corollary}\label{cor:upper:lasso}
    Suppose that Assumption \ref{assm:initial} holds, and let $s=\|\bfA_{0}\|_{0}$.
    For any $N>N_{0}(\epsilon_{1})$,
    with probability $1-(\epsilon_{0}+\epsilon_{1})/2$, the following three inequalities hold true simultaneously:
    \begin{align*}
        \frac{1}{N}\sum_{i=1}^{N}\left\|\left(\hat{\bfA}_{\Lasso}-\bfA_{0}\right)\bx_{i}\right\|_{L^{2}}^{2}&\le \frac{8s\lambda_{\Lasso}^{2}}{\kappa_{\min}}\left(1\vee \frac{\log(4\epsilon_{0}^{-1})}{s\log(2ed^{2}/s)}\right),\\
        \left\|\hat{\bfA}_{\Lasso}-\bfA_{0}\right\|_{2}^{2}&\le \frac{16s\lambda_{\Lasso}^{2}}{\kappa_{\min}^{2}}\left(1\vee \frac{\log(4\epsilon_{0}^{-1})}{s\log(2ed^{2}/s)}\right),\\
        \left\|\hat{\bfA}_{\Lasso}-\bfA_{0}\right\|_{1}&\le \frac{8s\lambda_{\Lasso}}{\kappa_{\min}}\left(1\vee \frac{\log(4\epsilon_{0}^{-1})}{s\log(2ed^{2}/s)}\right).
    \end{align*}
\end{corollary}
Hence, we obtain the rates of convergence of the Lasso estimation in several norms.
For instance, the rate of convergence of Lasso in $\ell^{2}$-norm is $\sqrt{s\log(d^{2}/s)/N}$ (by choosing the minimal $\lambda_{\Lasso}$).
In the discussion on lower bounds for errors below, we see that this rate is indeed optimal in a minimax sense.
Note that this result coincides with that of \citet{dexheimer2024lasso} except for $N$ (in our result) and $T$ (in their result).
Hence, as we can imagine, $N$ in repeated measurements plays the role of $T$ in long-term observations even under high-dimensional estimation problems.

\begin{remark}
    Due to the definition of $N_{0}(\epsilon_{1})$, which is adapted from \citet{dexheimer2024lasso}, our high-probability bounds are non-vacuous only under asymptotic regimes such that $s\ll N$ and $d\ll N$, while they allow $d^{2} \gg N$.
    This is seemingly restrictive but, in practice, quite mild.
    It is because most of the diagonal elements of $\bfA_{0}$ should be non-zero in many applications; that is, the evolution of each coordinate is dependent on itself.
    Under such self-dependence, $d\le s$ holds, and thus this restriction is not problematic in many situations.
    Note that the additional asymptotic setting $d\ll N$ simplifies the argument around eigenvalues.
    In particular, what we need to guarantee is only the ordinary non-degeneracy of the $d\times d$ sample second moment matrix, which can be derived by the elementary $\epsilon$-net argument \citep{vershynin2018high}.
    An original motivation of this approach by \citet{dexheimer2024lasso} is due to the proof strategy using the generic chaining argument for a martingale term, which depends on the largest eigenvalues of a sample second moment matrix rather than the largest diagonal element of it \citep[e.g., see][]{ciolek2020dantzig}.
    But the approach is useful for deriving a concise proof under the mild condition $d\ll N$, and thus we employ it.
\end{remark}

\begin{remark}
    From the perspective of theoretical statistics, one may ask whether $d\ll N$ is a removable assumption.
    We can answer to this question affirmatively by allowing slight deterioration in the rate of convergence from $s\log(d^{2}/s)/N$ to $s\log(d^{2})/N$ and introducing an effective rank of $\bSigma$ defined as $r(\bSigma)=\tr(\bSigma)/\|\bSigma\|_{\op}$; refer to Appendix \ref{append:r}.
    In total, the Lasso estimation converges under $s\log(d^{2})/N\to 0$ and $r(\bSigma)/N\to0$.
\end{remark}

We also derive the following proposition bounding the error of the Slope estimator.
\begin{proposition}\label{prop:upper:slope}
    Suppose that Assumption \ref{assm:initial} holds, and let $s=\|\bfA_{0}\|_{0}$.
    Fix a constant $c_{\Slope}\ge c_{\Lasso}\sqrt{\kappa^{\ast}}$.
    If $\lambda_{\Slope}$ satisfies the identity
    \begin{equation*}
        \lambda_{\Slope}=\frac{2c_{\Slope}}{\sqrt{N}},
    \end{equation*}
    then, for some $c\ge 1$ dependent only on $T$, $\|\bfA_{0}\|_{\op}$, $K$, and $\|\bSigma\|_{\op}$, for any $\epsilon_{0}\in(0,1)$ and $N\in\N$, with probability at least $1-\epsilon_{0}/2-2\times9^{d}(e^{-N}+\exp(-N\kappa_{\min}^{2}/(c(1+d/N))) + \exp(-N\kappa_{\min}^{2}/(c(4+\kappa_{\min}))))$, for all $\bfA\in\R^{d\times d}$ with $\|\bfA\|_{0}\le s$ simultaneously,
    \begin{align*}
        \frac{1}{N}\sum_{i=1}^{N}\left\|\left(\hat{\bfA}_{\Slope}-\bfA_{0}\right)\bx_{i}\right\|_{L^{2}}^{2}+\frac{2c_{\Slope}\left\|\bfA-\hat{\bfA}_{\Slope}\right\|_{\ast}}{\sqrt{N}}&\le \frac{1}{N}\sum_{i=1}^{N}\left\|\left(\bfA-\bfA_{0}\right)\bx_{i}\right\|_{L^{2}}^{2}\\
        &\quad+\frac{32c_{\Slope}^{2}\log(2ed^{2}/s)}{N\kappa_{\min}}\left(s\vee \frac{\log(4\epsilon_{0}^{-1})}{\log(2ed^{2}/s)}\right).
    \end{align*}
\end{proposition}

We derive the bounds for the errors of Slope under some norms.
\begin{corollary}\label{cor:upper:slope}
    Suppose that Assumption \ref{assm:initial} holds, and let $s=\|\bfA_{0}\|_{0}$.
    Let $\epsilon_{0},\epsilon_{1}\in(0,1)$.
    For any $N>N_{0}(\epsilon_{1})$, 
    with probability $1-(\epsilon_{0}+\epsilon_{1})/2$, the following inequalities hold true simultaneously:
    \begin{align*}
        \frac{1}{N}\sum_{i=1}^{N}\left\|\left(\hat{\bfA}_{\Slope}-\bfA_{0}\right)\bx_{i}\right\|_{L^{2}}^{2}&\le \frac{32sc_{\Slope}^{2}\log(2ed^{2}/s)}{N\kappa_{\min}}\left(1\vee \frac{\log(4\epsilon_{0}^{-1})}{s\log(2ed^{2}/s)}\right),\\
        \left\|\hat{\bfA}_{\Slope}-\bfA_{0}\right\|_{2}^{2}&\le \frac{64sc_{\Slope}^{2}\log(2ed^{2}/s)}{N\kappa_{\min}^{2}}\left(1\vee \frac{\log(4\epsilon_{0}^{-1})}{s\log(2ed^{2}/s)}\right),\\
        \left\|\hat{\bfA}_{\Slope}-\bfA_{0}\right\|_{\ast}&\le \frac{16sc_{\Slope}^{2}\log(2ed^{2}/s)}{\sqrt{N}\kappa_{\min}}\left(1\vee \frac{\log(4\epsilon_{0}^{-1})}{s\log(2ed^{2}/s)}\right),\\
        \left\|\hat{\bfA}_{\Slope}-\bfA_{0}\right\|_{1}&\le \frac{16sc_{\Slope}^{2}\log(2ed^{2}/s)}{\sqrt{N}\kappa_{\min}\log2}\left(1\vee \frac{\log(4\epsilon_{0}^{-1})}{s\log(2ed^{2}/s)}\right).
    \end{align*}
\end{corollary}
This corollary shows the rate of convergence of the Slope estimation.
For example, that in $\ell^{2}$-norm is $\sqrt{s\log(d^{2}/s)/N}$, which is shown to be minimax optimal below.

\subsubsection{Differences between Lasso and Slope}
Let us remark on the differences between the Lasso estimation and Slope estimation from the statistical and computational perspectives.

We first compare their statistical guarantees (Propositions \ref{prop:upper:lasso} and \ref{prop:upper:slope}) and observe an advantage of Slope.
While both the guarantees derive the same rate of convergence by selecting minimal regularization coefficients, the dependence structures of the minimal coefficients on the unknown sparsity level $s=\|\bfA_{0}\|_{0}$ are different.
Specifically, although the minimal $\lambda_{\Lasso}=2c_{\Lasso}\sqrt{\kappa^{\ast}\log(2ed^{2}/s)/N}$ contains the unknown sparsity level $s$, $\lambda_{\Slope}=2c_{\Lasso}\sqrt{\kappa^{\ast}}/\sqrt{N}$ is independent of $s$.
This independence is often referred to as a hopeful theoretical property of the Slope estimator \citep[e.g.,][]{bellec2018slope,dexheimer2024lasso}.

We also remark on the advantage of Lasso over Slope in computation.
The computational derivation of Lasso \eqref{eq:lasso:definition} is fast since the regularization $\|\cdot\|_{1}$ is \emph{separable}.
For example, the ISTA and FISTA algorithms \citep{beck2009fast} are typical for the optimization of the Lasso problem \eqref{eq:lasso:definition}, and they use the separability of the problem for efficient computation.
On the other hand, $\|\cdot\|_{\ast}$ in the Slope estimation \eqref{eq:slope:definition} is non-separable.
While several studies \citep{bu2019algorithmic,larsson2023coordinate} have analysed fast optimization of the Slope problem \eqref{eq:slope:definition}, its efficient computation is still in development in comparison to the Lasso problem \eqref{eq:lasso:definition}.

\subsection{Lower bounds for errors}
We show the minimax optimal rate of our problem.

\begin{proposition}\label{prop:lower}
    Let $d\ge 4$, $s\ge 2d$, $T=1$, and $\bxi_{i}=\zero_{d}$ for all $i=1,\ldots,n$.
    Fix a function $\ell:[0,\infty)\to[0,\infty)$ such that it is monotonically increasing, $\ell(0)=0$, and $\ell\not\equiv0$.
    For some constants $c,c'>0$ dependent only on $\ell$, it holds that
    \begin{equation}
        \inf_{\hat{\bfA}}\sup_{\substack{\bfA\in\R^{d\times d}\colon \\\|\bfA\|_{0}\le s_{0}}}\E_{\bfA}\left[\ell\left(c\psi_{\epsilon,p}^{-1}\left\|\hat{\bfA}-\bfA\right\|_{p}\right)\right]\ge c',\ \text{where}\ \psi_{\epsilon,p}:=s^{1/p}\sqrt{N^{-1}\log(ed^{2}/s)},
    \end{equation}
    where the infimum is taken over all estimators of $\bfA$.
\end{proposition}
Although we fix $\bxi_{i}=\zero_{d}$ for the sake of ease, we can generalize it to random vectors. 
For example, similar results under $\|\E[\bxi_{i}\bxi_{i}^\top]\|_{\op}\le a$ for arbitrary numerical constant $a>0$ follow from essentially identical proofs; see Appendix \ref{append:proofs:lower} for details.

Notably, for $p=2$ and $\ell(u)=u$, the rates of convergence of Lasso and Slope guaranteed by Corollaries \ref{cor:upper:lasso} and \ref{cor:upper:slope} coincide with the minimax rate derived by this proposition.
Hence, we can conclude that the Lasso and Slope estimators are minimax rate-optimal in $\ell^{2}$-norm.
We also remark that even for $p=1$, the rates of convergence by the corollaries also match the rate by this lower bound up to a logarithmic factor.

\section{Numerical experiments}\label{sec:numerical}
We exhibit numerical experiments of Lasso and Slope estimators of the high-dimensional drift coefficients of OU processes under i.i.d.~observations.
In particular, we consider large $d$ but small $s$ asymptotics; see Appendix \ref{append:numerical} for an experiment with large $N$ and fixed $d$ and $s$.

\subsection{Setup}
Our experiments consist of 10 iterations for each dimension setting $d=5,6,\ldots,24,25$.
For each iteration, we generate 500 i.i.d.~paths of OU processes $\{\bx_{i}\}_{i=1}^{500}$, where each $\bx_{i}$ is defined on the time interval $[0,1]$ (hence, $T=1$).
We set the training data $\{\bx_{i}\}_{i=1}^{400}$ (it means that $N=400$) and the validation data $\{\bx_{i}\}_{i=401}^{500}$ for selection of regularization coefficients.
Precisely speaking, $\lambda_{\Lasso}$ and $\lambda_{\Slope}$ are chosen by cross-validation defined as
\begin{align*}
    \lambda_{\Lasso}&\in\argmin_{\lambda\in\Lambda}\frac{1}{100}\sum_{i=401}^{500}\left(-\int_{0}^{1}\left(\hat{\bfA}_{\Lasso}(\lambda)\bx_{i}(t)\right)^{\top}\diff\bx_{i}(t)+\frac{1}{2}\int_{0}^{1}\left\|\hat{\bfA}_{\Lasso}(\lambda)\bx_{i}(t)\right\|_{2}^{2}\diff t\right),\\
    \lambda_{\Slope}&\in\argmin_{\lambda\in\Lambda}\frac{1}{100}\sum_{i=401}^{500}\left(-\int_{0}^{1}\left(\hat{\bfA}_{\Slope}(\lambda)\bx_{i}(t)\right)^{\top}\diff\bx_{i}(t)+\frac{1}{2}\int_{0}^{1}\left\|\hat{\bfA}_{\Slope}(\lambda)\bx_{i}(t)\right\|_{2}^{2}\diff t\right),
\end{align*}
where $\Lambda=\{10^{-8.00},10^{-7.90},10^{-7.80},\ldots,10^{-6.00}\}$ is the logarithmic grid, and
\begin{equation*}
    \hat{\bfA}_{\Lasso}(\lambda)\in\argmin_{\bfA\in\R^{d\times d}}\left(\Loss_{400}(\bfA)+\lambda\|\bfA\|_{1}\right),\ \hat{\bfA}_{\Slope}(\lambda)\in\argmin_{\bfA\in\R^{d\times d}}\left(\Loss_{400}(\bfA)+\lambda\|\bfA\|_{\ast}\right)
\end{equation*}
For each dimension setting, we use the same $\bfA$ among 10 iterations.
$\bfA$ is determined as follows: the diagonal elements of $\bfA$ are random variables following the uniform distribution on $[-1,1]$, and the off-diagonal elements are $0$ with probability $0.8$ and random variables following the uniform distribution on $[-0.5,0.5]$ otherwise.
For integration in loss functions and path generation, we set the discretization step $\delta=0.01$.
We employ the Euler--Maruyama scheme with stepsize $\delta$ to generate paths of OU processes.
\begin{remark}
    Though we discretize the loss functions in the numerical experiments, one may wonder whether the resulting estimators remain consistent under such discretization. Specifically, for fixed $\delta$ and $N\to\infty$, the discretisation error would dominate the statistical error, leading to inconsistency.
    In view of the recent study by \citet{amorino2025sampling} on long-term discrete observations of high-dimensional ergodic SDEs, we conjecture that there exist balance conditions among $d$, $N$, and $\delta$ guaranteeing the convergence of Lasso and Slope.
    Since the rigorous derivation requires a separate analysis, we leave it for future work.
\end{remark}

The computation is based on the \texttt{SLOPE} package in R \citep{larsson2025slope}.
Initial drafts of some simulation scripts were generated with assistance from ChatGPT 4.0 and then fully reviewed, tested, and revised by the author. All results reported use code that the author verified.
The source code and data are available at \url{https://github.com/snakakita/sparse-ou-with-iid-paths}.

\subsection{Results}
Figures \ref{fig:distance} and \ref{fig:heatmap} summarize the results of the experiment.
Figure \ref{fig:distance} shows the $d^{-2}$-scaled squared $\ell^{2}$ distances (top) and $d^{-2}$-scaled $\ell^{1}$ distances (bottom) in our experiment.
We observe that the Lasso and Slope successfully achieve lower errors in comparison to the MLE for every $d=5,\ldots,25$ on average.
Figure \ref{fig:heatmap} illustrates the heatmap of the true value of $\bfA$, MLE, Lasso, and Slope under $d=15$.
We notice that the MLE works well on the support of the true $\bfA_{0}$, but the performance outside the support is poor.
On the other hand, the Lasso estimation and Slope estimation result in more precise estimation of most zero elements in the true $\bfA$ while keeping the performance on the support.
\begin{figure}[p]
    \centering
    \includegraphics[width=0.9\linewidth]{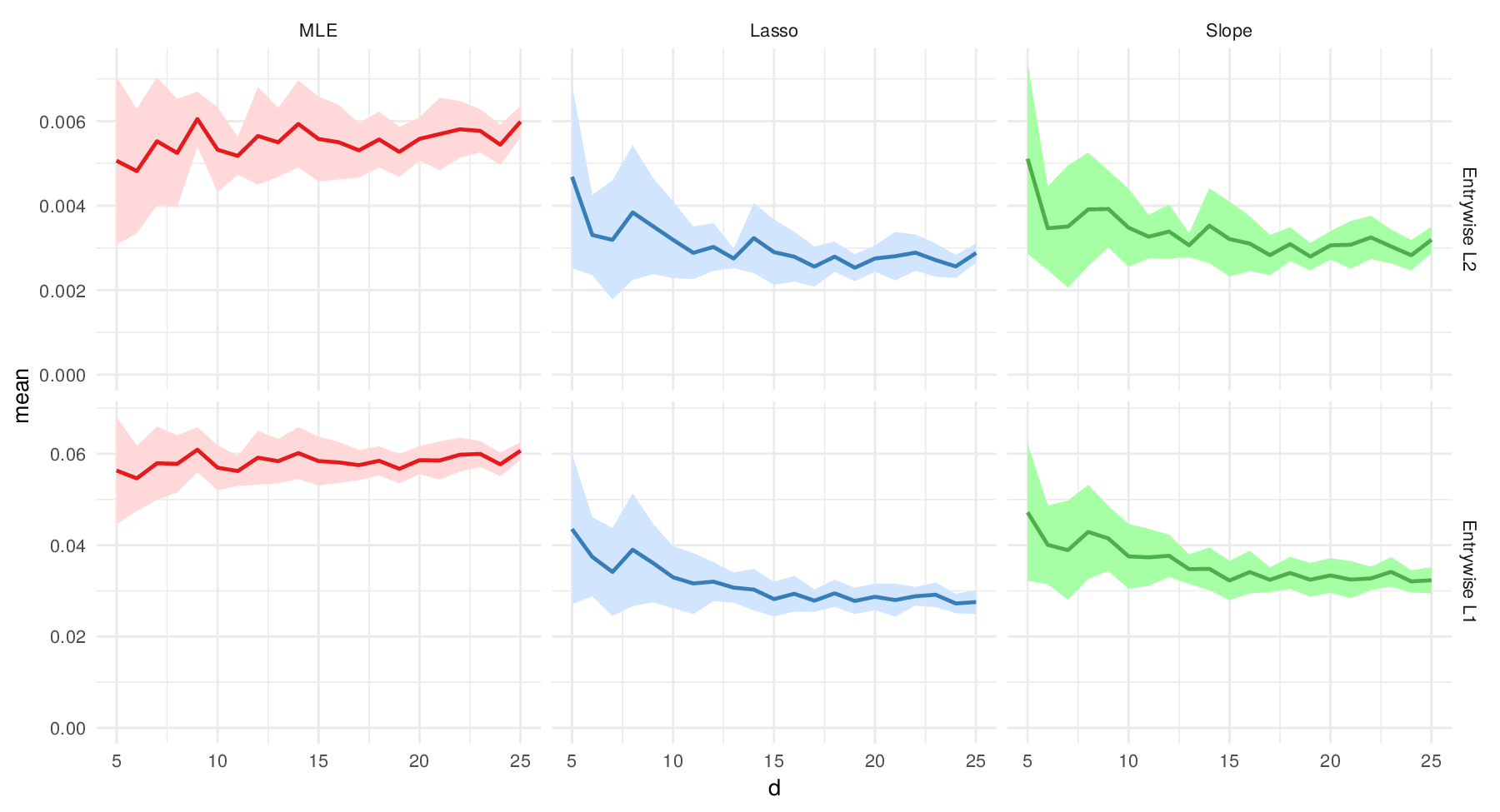}
    \caption{The comparison of $d^{-2}$-scaled squared $\ell^{2}$ distances (top) and $d^{-2}$-scaled $\ell^{1}$ distances (bottom) between the true value and the MLE (left), Lasso (middle), and Slope (right).
    The coloured areas represent (mean) $\pm$ (standard deviation).
    }
    \label{fig:distance}
    
    \includegraphics[width=0.9\linewidth]{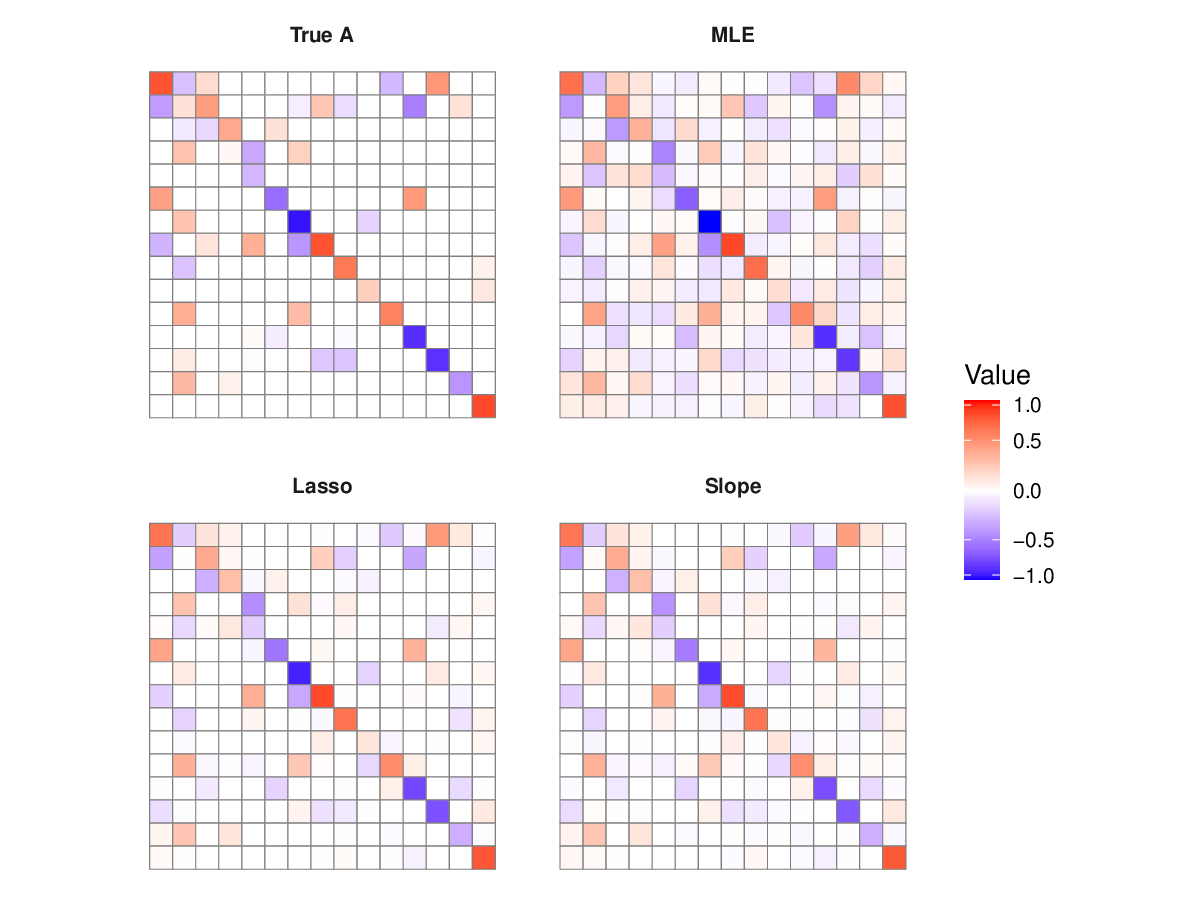}
    \caption{The heatmaps of the true value (top left), MLE (top right), Lasso (bottom left), and Slope (bottom right).
    The scale is adjusted to intensify the magnitude of small but non-zero values.}
    \label{fig:heatmap}
\end{figure}

\section{Concluding remarks}\label{sec:conclusion}
We have discussed sparse estimation of the drift parameter of high-dimensional Ornstein--Uhlenbeck processes given repeated measurements of i.i.d.~paths.
Similar to previous studies on the same problem with long-term observations under ergodicity, this study has shown the minimax rate optimality of sparse estimation methods.
In particular, the rates of convergence are formally identical to those of previous studies except for the terminal $T>0$, which is replaced with the sample size $N$ in our setup.
Although the derived result itself is an important contribution as a statistical theory, it is also important in the sense that there are many different directions for further development.

Let us remark on some concrete potential directions of future studies in line with this study.
(i) Estimation under the nonlinearity of drift coefficients in parameters and state variables is an important direction.
As \citet{ciolek2025lasso} establish sparse estimation of nonlinear parameters in ergodic SDEs by employing isoperimetry in Wiener spaces and the generic chaining approach of \citet{dirksen2015tail}, we also expect that a similar argument should hold under repeated measurements.
Recently, \citet{bayrakter2026sparse} investigated sparse estimation under discrete observations of i.i.d.~paths with nonlinearity in state variables.
Hence, an unsolved significant problem in this topic is nonlinearity in parameters.
(ii) Sparse estimation of high-dimensional interacting particle systems is a significant extension of the discussion.
Since statistical estimation of interacting particle systems has gathered intense research efforts \citep{amorino2023parameter,della2023lan,nickl2025bayesian}, we are motivated to develop high-dimensional estimation theories for such systems in order to prepare useful statistical methods.
(iii) Estimation of high-dimensional SDEs with mixed effects is also a hopeful direction.
Since longitudinal data, one of the motivations for estimation with repeated measurements, are sometimes naturally explained via mixed effects \citep{ditlevsen2005mixed,delattre2018parametric,delattre2025quasi}, high-dimensional estimation theories of SDEs with mixed effects are important.
(iv) Estimation of SDEs under discrete observations is critical in practice.
Recently, \citet{amorino2025sampling} study the Lasso estimation for ergodic diffusion processes with high-dimensional sparse parameters and long-term discrete observations.
Its extension to the i.i.d.~path observations is a future work.
(v) Introducing non-sparse inductive bias in dependent models is a significant and topical problem \citep{nakakita2026benign,luo2025roti,atanasov2025risk,moniri2025asymptotics,esmailimallory2025universality,jones2025generalisation}.
Since estimation with high-dimensional SDEs under non-sparse inductive bias (such as ridge estimators) has not been discussed, the community should be interested in this direction.
For example, as we introduce an effective rank $r(\bSigma)=\tr(\bSigma)/\|\bSigma\|_{\op}$ in Appendix \ref{append:r}, we may also consider non-sparse effective rank of $\bfA$.

\appendix
\section{Proofs of preliminary results}\label{append:preliminary}

We obtain a representation of the (scaled) log-likelihood of $\bfA$ as follows:
\begin{align*}
    \calL(\bfA)&=-\frac{1}{N}\log\frac{\bbP_{\bfA}^{N,T}}{\bbP_{\bfO}^{N,T}}(\{\bx_{i}(t)\colon i\in\{1,...,N\},t\in[0,T]\})\\
    &=\frac{1}{N}\sum_{i=1}^{N}\left(-\int_{0}^{T}\left(\bfA\bx_{i}(t)\right)^{\top}\diff \bx_{i}(t)+\frac{1}{2}\int_{0}^{T}\left(\bfA\bx_{i}(t)\right)^{\top}\left(\bfA\bx_{i}(t)\right)\diff t\right)\\
    &=\tr\left(\frac{1}{2}\bfA\hat{\bfC}_{N}\bfA^{\top}-\bfA_{0}\hat{\bfC}_{N}\bfA^{\top}-\bfM_{N}(T)\bfA^{\top}\right),
\end{align*}
where $\hat{\bfC}_{N}$ is the sample second moment matrix defined as
\begin{equation*}
    \hat{\bfC}_{N}:=\frac{1}{N}\sum_{i=1}^{N}\int_{0}^{T}\bx_{i}(t)^{\otimes2}\diff t,
\end{equation*}
and $\{\bfM_{N}(t)\colon t\in[0,T]\}$ is a matrix-valued martingale defined as
\begin{equation*}
    \bfM_{N}(t)=\frac{1}{N}\sum_{i=1}^{N}\int_{0}^{t}\diff\bw_{i}(s)(\bx_{i}(s))^{\top}.
\end{equation*}
In this Appendix \ref{append:preliminary}, we discuss concentration bounds for $\hat{\bfC}_{N}$ and $\bfM_{N}$.

\subsection{Fundamental decomposition}
We have the following exact decomposition:
\begin{equation*}
    \bx_{i}(t)=\bx_{i}^{0}(t)+\bx_{i}^{1}(t),
\end{equation*}
where $\bx_{i}^{0}(t)=\exp(t\bfA_{0})\bxi_{i}$ is a dynamics which transports the random initial value $\bxi_{i}$, and $\bx_{i}^{1}(t)=\int_{0}^{t}\exp((t-t')\bfA_{0})\diff \bw_{i}(t')$ is a centred stochastic process. 
Using this decomposition and the It\^{o} isometry, we have
\begin{align*}
    \E[\hat{\bfC}_{N}|\{\bxi_{i}\}]&=\frac{1}{N}\sum_{i=1}^{N}\int_{0}^{T}\left((\exp(t\bfA_{0})\bxi_{i})^{\otimes2}+\E\left[\left(\int_{0}^{t}\exp((t-t')\bfA_{0})\diff \bw_{i}(t')\right)^{\otimes2}\right]\right)\diff t\\
    &=\frac{1}{N}\sum_{i=1}^{N}\int_{0}^{T}\left((\exp(t\bfA_{0})\bxi_{i})^{\otimes2}+\int_{0}^{t}\exp((t-t')\bfA_{0})^{\otimes2}\diff t'\right)\diff t.
\end{align*}

\subsection{Eigenvalue behaviours}
We examine the behaviours of the eigenvalues of $\hat{\bfC}_{N}$ via Malliavin calculus and high-probability arguments.

We give a list of notation used within this section.
Let $\HS$ denote a real separable Hilbert space and $B=\{B(h)\colon h\in\HS\}$ be an isonormal Gaussian process over $\HS$, that is, a family of real centred Gaussian random variables with covariance $\E[B(h)B(h')]=\langle h,h'\rangle_{\HS}$ for any $h,h'\in\HS$.

Let us define
\begin{align*}
    \tilde{\bfC}_{N}&:=\frac{1}{N}\sum_{i=1}^{N}\int_{0}^{T}\left(\left(\exp(t\bfA_{0})\bxi_{i}\right)^{\otimes2}+\int_{0}^{t}\exp((t-t')\bfA_{0})^{\otimes2}\diff t\right)\diff t=\E[\hat{\bfC}_{N}|\{\bxi_{i}\}].
\end{align*}
We have
\begin{equation}
    \hat{\bfC}_{N}-\tilde{\bfC}_{N}=\frac{1}{N}\sum_{i=1}^{N}\left(2\int_{0}^{T}\bx_{i}^{0}(t)(\bx_{i}^{1}(t))^{\top}\diff t+\int_{0}^{T}\left(\bx_{i}^{1}(t)(\bx_{i}^{1}(t))^{\top}-\E\left[\bx_{i}^{1}(t)(\bx_{i}^{1}(t))^{\top}\right]\right)\diff t\right).
\end{equation}

\begin{proposition}[pointwise concentration]\label{prop:pointwise:d}
    Under Assumption \ref{assm:initial}, there exists a positive constant $c\ge1$ dependent only on $T$, $\|\bfA_0\|_{\op}$, $K$, and $\|\bSigma\|_{\op}$ such that for all $u>0$,
    \begin{equation}
        \sup_{\bv\in\bbS^{d-1}}\Pr\left(\left|\bv^{\top}\left(\hat{\bfC}_{N}-\tilde{\bfC}_{N}\right)\bv\right|\ge u\right)\le 2e^{-N}+2\exp\left(-\frac{Nu^{2}}{c(1+d/N)}\right) + 2\exp\left(-\frac{1}{c}\frac{Nu^{2}}{1+u}\right).
    \end{equation}
\end{proposition}

The statement and proof are mostly similar to those of \citet{ciolek2020dantzig}, but we have an additional tail bound for a first-order Wiener chaos due to the non-stationarity and non-centricity of $\bx_{i}(t)$.
Fortunately, the additional bound is negligible under $N\gg d$; therefore, our problem can be reduced to those of \citet{gaiffas2019sparse} and \citet{ciolek2020dantzig} from a high-level point of view.

We introduce notation for the proof: $B$ to be an isonormal Gaussian process on $\HS=L^{2}([0,T];\R^{Nd})$ such that for any $h\in\HS$,
\begin{equation*}
    B(f)=\int_{0}^{T}h(s)^{\top}\diff\bw(s),
\end{equation*}
where $\bw(t)=[\bw_{1}(t)^{\top}\cdots\bw_{N}(t)^{\top}]^{\top}$ is an $Nd$-dimensional standard Wiener process.
It clearly holds that $\E[B(h_{1})B(h_{2})]=\langle h_{1},h_{2}\rangle_{\HS}=\int_{0}^{T}h_{1}(s)^{\top}h_{2}(s)\diff s$ for any $h_{1},h_{2}\in\HS$.
We define the set of all smooth cylindrical random variables in the form $F=f(B(h_{1}),\ldots,B(h_{n}))$, where $n\in\N$ is arbitrary and $f$ is of $\calC^{\infty}$ such that $f$ and its derivatives are of polynomial growth.
The Malliavin derivative $DF$ of $F$ is defined as
\begin{equation*}
    D_{t}F=\sum_{i=1}^{n}\frac{\partial f}{\partial x_{i}}(B(h_{1}),\ldots,B(h_{n})))h_{i}(t).
\end{equation*}
We let $\bbD^{1,2}$ denote the closure of $\calS$ with respect to the norm $\|F\|_{1,2}^{2}=\E[F^{2}]+\E[\|DF\|_{\bbH}^{2}]$.

\begin{proof}
Let us fix a unit vector $\bv\in\bbS^{d-1}$ and define a matrix $\bfK_{v}=\bfI_{N}\otimes \bv^{\otimes2}\in\R^{Nd\times Nd}$.
Let $\bX^{\ell}(t)$ be 
\begin{equation*}
    \bX_{1:N}^{0}(t)=\left[\begin{matrix}
        \bx_{1}^{0}(t)\\
        \vdots\\
        \bx_{N}^{0}(t)
    \end{matrix}\right],\quad\bX_{1:N}^{1}(t)=\left[\begin{matrix}
        \bx_{1}^{1}(t)\\
        \vdots\\
        \bx_{N}^{1}(t)
    \end{matrix}\right],\quad\bX_{1:N}(t)=\bX_{1:N}^{0}(t)+\bX_{1:N}^{1}(t);
\end{equation*}
note that $\bX_{1:N}^{0}\in L^{2}([0,T];\R^{Nd})$ and $\bX_{1:N}^{1}(t)$ has a representation with $B$ such that for any $\br=[\br_{1}^{\top}\,\cdots\,\br_{N}^{\top}]^{\top}\in\R^{dN}$ ($\br_{i}\in\R^{d}$),
\begin{align*}
    &(\br^{\top}\bX_{1:N}^{1}(t))
    =\int_{0}^{t}\sum_{i=1}^{N}\br_{i}^{\top}\exp\left((t-t')\bfA_{0}\right)\diff \bw_{i}(t')\\
    &=\int_{0}^{t}\br^{\top}\left(\bfI_{N}\otimes\exp\left((t-t')\bfA_{0}\right)\right)\diff \bw(t')
    =B\left(\ind_{[0,t]}(\cdot)\left(\bfI_{N}\otimes\exp\left((t-t')\bfA_{0}\right)\right)^{\top}\br\right).
\end{align*}
Since for all $\bv\in\bbS^{d-1}$,
\begin{align*}
    &\Pr\left(\left|\bv^{\top}\left(\hat{\bfC}_{N}-\tilde{\bfC}_{N}\right)\bv\right|\ge u\right)\\
    &=\Pr\left(\left|\int_{0}^{T}\left\{\left(\bX_{1:N}(t)\right)^{\top}\bfK_{\bv}\left(\bX_{1:N}(t)\right)-\E\left[\left(\bX_{1:N}(t)\right)^{\top}\bfK_{\bv}\left(\bX_{1:N}(t)\right)\right]\right\}\diff t\right|\ge Nu\right)\\
    &\le \Pr\left( \left|2\int_{0}^{T}\left(\bX_{1:N}^{0}(t)\right)^{\top}\bfK_{\bv}\left(\bX_{1:N}^{1}(t)\right)\diff t\right|\ge Nu/2\right)\\
    &\quad+\Pr\left( \left|\int_{0}^{T}\left\{\left(\bX_{1:N}^{1}(t)\right)^{\top}\bfK_{\bv}\left(\bX_{1:N}^{1}(t)\right)-\E\left[\left(\bX_{1:N}^{1}(t)\right)^{\top}\bfK_{\bv}\left(\bX_{1:N}^{1}(t)\right)\right]\right\}\diff t\right|\ge Nu/2\right),
\end{align*}
we separately evaluate the first term (first-order Wiener chaos) and second term (second-order Wiener chaos) on the right-hand side.

(Step i) First, let us consider a bound for the first-order Wiener chaos.
Let us define the event
\begin{equation*}
    E_{1}=\left\{\begin{array}{c}
        \bw\in \calC([0,T];\R^{Nd}),\bxi=[\bxi_{1}^{\top}\cdots\bxi_{N}^{\top}]^{\top}\in\R^{Nd}\colon\\
        \left|2\int_{0}^{T}\left(\bX_{1:N}^{0}(t)\right)^{\top}\bfK_{\bv}\left(\bX_{1:N}^{1}(t)\right)\diff t\right|\ge Nu/2 
    \end{array}\right\}.
\end{equation*}
What we need to obtain is an upper bound on $\int_{\R^{dN}}\int_{\calC([0,T];\R^{dN})}\ind_{E_{1}}(\bw,\bxi)\Pr(\diff\bw)\Pr(\diff\bxi)$ (recall the independence between $\{\bw_{i}\}$ and $\{\bxi_{i}\}$).
We begin by bounding the following conditional probability given the initial values: for fixed $\{\bxi_{i}\}\subset\R^{d}$,
\begin{equation*}
    \Pr\left( \left.\left|2\int_{0}^{T}\left(\bX_{1:N}^{0}(t)\right)^{\top}\bfK_{\bv}\left(\bX_{1:N}^{1}(t)\right)\diff t\right|\ge Nu/2\right|\{\bxi_{i}\}\right)\left(=\int_{\calC([0,T];\R^{dN})}\ind_{E_{1}}(\bw,\bxi)\Pr(\diff\bw)\right).
\end{equation*}
We define
\begin{equation*}
    Z^{1}(\bv):=\int_{0}^{T}\left(\bX_{1:N}^{0}(t)\right)^{\top}\bfK_{\bv}\left(\bX_{1:N}^{1}(t)\right)\diff t.
\end{equation*}
Since $\bfK_{\bv}\bX_{1:N}^{0}(\cdot)\in\HS$, $Z^{1}(\bv)=\langle \bfK_{\bv}\bX_{1:N}^{0}(\cdot),\bX_{1:N}^{1}(\cdot)\rangle_{\HS}$.
Proposition 3.4.3 of \citet{nualart2018introduction} yields that for all $u\in[0,T]$,
\begin{align*}
    \left(D_{u}Z^{1}(\bv)\right)^{\top}=\int_{u}^{T}\left(\bX_{1:N}^{0}(s)\right)^{\top}\bfK_{\bv}\left(\bfI_{N}\otimes\exp\left((s-u)\bfA_{0}\right)\right)\diff s.
\end{align*}
As $Z^{1}(\bv)$ is a first-order Wiener chaos \citep[Eq.~(4.6) of][]{nualart2018introduction}, it holds that $Z^{-1}(\bv)\in\bbD^{1,2}$, $L^{-1}Z_{1,\bv}=-Z_{1,\bv}$ ($L$ is the generator of the Ornstein--Uhlenbeck semigroup), and
\begin{align*}
    &\langle D Z^{1}(\bv),-DL^{-1} Z^{1}(\bv)\rangle_{\HS}\\
    &=\left\| DZ^{1}(\bv)\right\|_{\HS}^{2}=\int_{0}^{T}(D_{t}Z^{1}(\bv))^{\top}(D_{t}Z^{1}(\bv))\diff t\\
    &=\int_{0}^{T}\sum_{i=1}^{N}\left\|\int_{t}^{T}\left(\bX_{i}^{0}(s)\right)^{\top}\bv^{\otimes2} \exp\left((s-t)\bfA_{0}\right)\diff s\right\|^{2}\diff t\\
    &\le \int_{0}^{T}(T-t)\sum_{i=1}^{N}\int_{t}^{T}\left\|\left(\bX_{i}^{0}(s)\right)^{\top}\bv^{\otimes2} \exp\left((s-t)\bfA_{0}\right)\right\|_{2}^{2}\diff s\diff t\\
    &= \int_{0}^{T}(T-t)\sum_{i=1}^{N}\int_{t}^{T}\left(\left(\bX_{i}^{0}(s)\right)^{\top}\bv\right)^{2}\left\|\bv^{\top} \exp\left((s-t)\bfA_{0}\right)\right\|_{2}^{2}\diff s\diff t\\
    &\le \sup_{t\in[0,T]}\left(\sum_{i=1}^{N}\left(\left(\bX_{i}^{0}( t)\right)^{\top}\bv\right)^{2}\right)\int_{0}^{T}(T-t)\int_{t}^{T}\left\|\bv^{\top}\exp\left((s-t)\bfA_{0}\right)\right\|_{2}^{2}\diff s\diff t\\
    &\le \left\|\sum_{i=1}^{N}\bxi_{i}\bxi_{i}^{\top}\right\|_{\op}e^{2T\|\bfA_0\|_{\op}}\int_{0}^{T}(T-t)\int_{t}^{T}\exp\left(2(s-t)\left\|\bfA_{0}\right\|_{\op}\right)\diff s\diff t\\
    &\le c\left\|\sum_{i=1}^{N}\bxi_{i}\bxi_{i}^{\top}\right\|_{\op},
\end{align*}
where $c>0$ is a positive constant dependent only on $T$ and $\|\bfA_0\|_{\op}$, and we used the Cauchy--Schwarz inequality for the first inequality and the following estimate for the last inequality:
\begin{equation*}
    \sum_{i=1}^{N}\left(\left(\bX_{i}^{0}( t)\right)^{\top}\bv\right)^{2}=\bv^{\top}\exp(t\bfA_{0})\left(\sum_{i=1}^{N}\bxi_{i}\bxi_{i}^{\top}\right)\exp(t\bfA_{0})^{\top}\bv\le \left\|\sum_{i=1}^{N}\bxi_{i}\bxi_{i}^{\top}\right\|_{\op}e^{2T\|\bfA_0\|_{\op}}.
\end{equation*}
Hence,  Theorem 4.1 of \citet{nourdin2009density} yields a concentration bound such that for some $c\ge1$, for all $u\ge0$ and $\{\bxi_{i}\}$,
\begin{equation*}
    \Pr\left( \left.\left|2\int_{0}^{T}\left(\bX_{1:N}^{0}(t)\right)^{\top}\bfK_{\bv}\left(\bX_{1:N}^{1}(t)\right)\diff t\right|\ge Nu/2\right|\{\bxi_{i}\}\right)\le 2\exp\left(-\frac{N^{2}u^{2}}{c\left\|\sum_{i=1}^{N}\bxi_{i}\bxi_{i}^{\top}\right\|_{\op}}\right).
\end{equation*}
Using the estimate 
\begin{align}
    \sup_{\bv\in\bbS^{d-1}}\left|\bv^{\top}\left(\tilde{\bfC}_{N}-\bfC_{\infty}\right)\bv\right|&=\sup_{\bv\in\bbS^{d-1}}\left|\frac{1}{N}\sum_{i=1}^{N}\int_{0}^{T}\bv^{\top}\exp(t\bfA_{0})\left(\bxi_{i}\bxi_{i}^{\top}-\bSigma\right)\exp(t\bfA_{0})^{\top}\bv\diff t\right|\notag\\
    &\le Te^{2T\|\bfA_0\|_{\op}}\sup_{\bv\in\bbS^{d-1}}\left|\frac{1}{N}\sum_{i=1}^{N}\bv^{\top}\left(\bxi_{i}\bxi_{i}^{\top}-\bSigma\right)\bv\right|\label{eq:covariance:initials},
\end{align}
we obtain a classical estimate for the sample covariance matrix \citep[e.g., see Section 4.7 of][]{vershynin2018high} 
for some $c\ge1$ dependent only on $T$, $\|\bfA_0\|$, $K$, and $\|\bSigma\|_{\op}$,
\begin{equation}
    \Pr(E_{2}):=\Pr\left(\left\|\frac{1}{N}\sum_{i=1}^{N}\bxi_{i}\bxi_{i}^{\top}\right\|_{\op}\le c(1+\sqrt{d/N}+d/N)\right)\ge 1-2e^{-N}. \label{eq:tail:initials:d}
\end{equation}
Hence,
\begin{align*}
    &\Pr\left(\left|2\int_{0}^{T}\left(\bX_{1:N}^{0}(t)\right)^{\top}\bfK_{\bv}\left(\bX_{1:N}^{1}(t)\right)\diff t\right|\ge Nu/2\right)\\
    &=\int_{E_{2}\cup E_{2}^{c}}\int_{\calC([0,T];\R^{Nd})}\ind_{E_{1}}(\bw,\bxi)\Pr(\diff\bw)\Pr(\diff\bxi)\\
    &=\int_{E_{2}\cup E_{2}^{c}} \Pr\left( \left.\left|2\int_{0}^{T}\left(\bX_{1:N}^{0}(t)\right)^{\top}\bfK_{\bv}\left(\bX_{1:N}^{1}(t)\right)\diff t\right|\ge Nu/2\right|\{\bxi_{i}\}\right)\Pr(\diff \bxi_{1})\times \cdots\times \Pr(\diff \bxi_{N})\\
    &\le 2\exp\left(-\frac{Nu^{2}}{c(1+d/N)}\right)+2e^{-N}.
\end{align*}

(Step ii) We define 
\begin{equation*}
    Z^{2}(\bv):=\int_{0}^{T}\left\{\left(\bX_{1:N}^{1}(s)\right)^{\top}\bfK_{\bv}\left(\bX_{1:N}^{1}(s)\right)-\E\left[\left(\bX_{1:N}^{1}(s)\right)^{\top}\bfK_{\bv}\left(\bX_{1:N}^{1}(s)\right)\right]\right\}\diff s.
\end{equation*}
By the chain rule \citep[Proposition 3.3.2 of][]{nualart2018introduction}, we have
\begin{equation*}
    (D_{u}Z^{2}(\bv))^{\top}=2\int_{u}^{T}\left(\bX_{1:N}^{1}(s)\right)^{\top}\bfK_{\bv}\left(\bfI_{N}\otimes\exp\left((s-u)\bfA_{0}\right)\right)\diff s.
\end{equation*}
Since $Z^{2}(\bv)$ is a second-order Wiener chaos \citep{nualart2018introduction}, it holds that $Z^{2}(\bv)\in\bbD^{1,2}$, $L^{-1}Z^{2}(\bv)=-Z^{2}(\bv)/2$, and for some positive constant $c>0$ dependent only on $T$ and $\|\bfA_0\|_{\op}$,
\begin{align*}
    &\langle DZ^{2}(\bv),-DL^{-1}Z^{2}(\bv)\rangle\\
    &=\frac{1}{2}\left\|DZ^{2}(\bv)\right\|_{\HS}^{2}\\
    &=2\int_{0}^{T}\left\|\int_{t}^{T}\left(\bX_{1:N}^{1}(s)\right)^{\top}\bfK_{\bv}\left(\bfI_{N}\otimes \exp\left((s-t)\bfA_{0}\right)\right)\diff s\right\|_{2}^{2}\diff t\\
    &\le 2\int_{0}^{T}(T-t)\int_{t}^{T}\left\|\left(\bX_{1:N}^{1}(s)\right)^{\top}\bfK_{\bv}\left(\bfI_{N}\otimes \exp\left((s-t)\bfA_{0}\right)\right)\right\|_{2}^{2}\diff s\diff t\\
    &\le 2\int_{0}^{T}(T-t)\int_{t}^{T}\left\|\bfK_{\bv}\bX_{1:N}^{1}(s)\right\|_{2}^{2}\exp\left(2(s-t)\|\bfA_{0}\|_{\op}\right)\diff s\diff t\\
    &= 2\int_{0}^{T}(T-t)\int_{t}^{T}\left(\bX_{1:N}^{1}(s)\right)^{\top}\bfK_{\bv}\bX_{1:N}^{1}(s)\exp\left(2(s-t)\|\bfA_{0}\|_{\op}\right)\diff s\diff t\\
    &\le 2\int_{0}^{T}(T-t)\exp\left(2T\|\bfA_{0}\|_{\op}\right)\diff t\int_{0}^{T}\left(\bX_{1:N}^{1}(s)\right)^{\top}\bfK_{\bv}\bX_{1:N}^{1}(s)\diff s\\
    &\le c\left(Z^{2}(\bv)+Nc\right),
\end{align*}
where the second inequality is given by Jensen's inequality, and the last inequality is due to the following bound:
\begin{align*}
    &\int_{0}^{T}\E\left[\left(\bX_{1:N}^{1}(s)\right)^{\top}\bfK_{\bv}\left(\bX_{1:N}^{1}(s)\right)\right]\diff s=\sum_{i=1}^{N}\int_{0}^{T}\E\left[\left(\left(\bX_{i}^{1}(s)\right)^{\top}\bv\right)^{2}\right]\diff s\\
    &=N\int_{0}^{T}\int_{0}^{s}\bv^{\top}\exp\left((s-u)\bfA_{0}\right)^{\otimes2}\bv\diff u\diff s
    \le N\int_{0}^{T}\int_{0}^{s}\exp\left(2(s-u)\|\bfA_{0}\|_{\op}\right)\diff u\diff s\\
    &\le Nc.
\end{align*}
Hence, Theorem 4.1 of \citet{nourdin2009density} yields the second bound.
Thus, we obtain the conclusion.
\end{proof}

\begin{lemma}[uniform concentration]\label{lem:uniformbound}
    Under Assumption \ref{assm:initial}, there exists a positive constant $c\ge1$ dependent only on $T$, $\|\bfA_0\|_{\op}$, $K$, and $\|\bSigma\|_{\op}$ such that for any $u\ge0$, we have
    \begin{equation*}
        \Pr\left(\sup_{\bv\in\bbS^{d-1}}\left|\bv^{\top}\left(\hat{\bfC}_{N}-\tilde{\bfC}_{N}\right)\bv\right|\ge u\right)\le 2\times9^{d}\left(e^{-N}+\exp\left(-\frac{Nu^{2}}{c(1+d/N)}\right) + \exp\left(-\frac{Nu^{2}}{c(2+u)}\right)\right).
    \end{equation*}
\end{lemma}

\begin{proof}
    We take $\calN$, a $(1/4)$-net of $\bbS^{d-1}$, and derive
\begin{equation*}
    \left\|\left(\hat{\bfC}_{N}-\tilde{\bfC}_{N}\right)\right\|_{\op}\le 2\sup_{\bv\in\calN}\left|\bv^{\top}\left(\hat{\bfC}_{N}-\tilde{\bfC}_{N}\right)\bv\right|
\end{equation*}
(see Lemma S.8 of \citealp{bartlett2020benign}).
As we can take $\calN$ satisfying $|\calN|\le 9^{d}$ \citep{vershynin2018high}, Proposition \ref{prop:pointwise:d} yields that
\begin{align*}
    &\Pr\left(\sup_{\bv\in\bbS^{d-1}}\left|\bv^{\top}\left(\hat{\bfC}_{N}-\tilde{\bfC}_{N}\right)\bv\right|\ge u\right)\\
    &\le \Pr\left(\sup_{\bv\in\calN}\left|\bv^{\top}\left(\hat{\bfC}_{N}-\tilde{\bfC}_{N}\right)\bv\right|\ge u/2\right)\le \sum_{\bv\in\calN}\Pr\left(\left|\bv^{\top}\left(\hat{\bfC}_{N}-\tilde{\bfC}_{N}\right)\bv\right|\ge u/2\right)\\
    &\le 2\times9^{d}\left(e^{-N} +\exp\left(-\frac{Nu^{2}}{4c(1+d/N)}\right) + \exp\left(-\frac{Nu^{2}}{2c(2+u)}\right)\right).
\end{align*}
    This is the desired conclusion.
\end{proof}

Let us define an event:
\begin{equation*}
    Q_{N}:=\left\{\sup_{\bv\in\bbS^{d-1}}\left|\bv^{\top}\left(\hat{\bfC}_{N}-\tilde{\bfC}_{N}\right)\bv\right|\le \frac{\kappa_{\min}}{4}\right\}\cap \left\{\sup_{\bv\in\bbS^{d-1}}\left|\bv^{\top}\left(\tilde{\bfC}_{N}-\bfC_{\infty}\right)\bv\right|\le \frac{\kappa_{\min}}{4}\right\}.
\end{equation*}
Before giving high-probability bounds on the eigenvalues of $\hat{\bfC}_{N}$, we remark on each event on the right hand side.
The first event on the right-hand side is a concentration with respect to the randomness of Wiener processes, and the second event is a concentration with respect to the randomness of initial values, which can be evaluated by a classical argument \citep[e.g.,][]{vershynin2018high}.

\begin{proposition}[high-probability bounds on eigenvalues]\label{prop:spectra}
Under Assumption \ref{assm:initial}, there exists a positive constant $c\ge1$ dependent only on $T$, $\|\bfA_0\|_{\op}$, $K$, and $\|\bSigma\|_{\op}$ such that
\begin{align*}
    &\Pr\left(\left\{\lambda_{\min}(\hat{\bfC}_{N})\ge \frac{\kappa_{\min}}{2}\right\}\cup\left\{\lambda_{\max}(\hat{\bfC}_{N})\le \frac{\kappa_{\min}}{2}+\kappa_{\max}\right\}\right)\\
    &\ge \Pr(Q_{N})\ge 1-8\times9^{d}\times\exp\left(-\frac{N}{c(1+d/N)}\right).
\end{align*}
\end{proposition}

\begin{proof}
Let us suppose that $Q_{N}$ occurs.
Note that for some $c\ge 1$ dependent only on $T$, $\|\bfA_{0}\|_{\op}$, $K$, and $\|\bSigma\|_{\op}$, $\Pr(Q_{N})\ge 1-2\times 9^{d}(2e^{-N}+\exp(-(N\kappa_{\min}^{2})/(64c(1+d/N))) + \exp(-(N\kappa_{\min}^{2})/(c(64+16\kappa_{\min}))))$ by Lemma \ref{lem:uniformbound} and the following bound \citep[see][]{vershynin2018high}: for all $u\ge0$,
\begin{align*}
    \Pr\left(\sup_{\bv\in\bbS^{d-1}}\left|\bv^{\top}\left(\tilde{\bfC}_{N}-\bfC_{\infty}\right)\bv\right|\ge u\right)&\le \Pr\left(\sup_{\bv\in\bbS^{d-1}}\left|\bv^{\top}\left(\frac{1}{N}\sum_{i=1}^{N}\bxi_{i}\bxi_{i}^{\top}-\bSigma\right)\bv\right|\ge \frac{u}{c}\right)\\
    &\le 2\times 9^{d}e^{-\frac{N}{c}\min\{u,u^{2}\}},
\end{align*}
where we used the estimate \eqref{eq:covariance:initials}.
Therefore, we have
\begin{align*}
    \lambda_{\min}(\hat{\bfC}_{N})\ge \lambda_{\min}(\bfC_{\infty})-\sup_{\bv\in\bbS^{d-1}}\left|\bv^{\top}\left(\hat{\bfC}_{N}-\bfC_{\infty}\right)\bv\right|\ge \frac{\kappa_{\min}}{2},
\end{align*}
and 
\begin{align*}
    \lambda_{\max}(\hat{\bfC}_{N})\le \lambda_{\max}(\bfC_{\infty})+\sup_{\bv\in\bbS^{d-1}}\left|\bv^{\top}\left(\hat{\bfC}_{N}-\bfC_{\infty}\right)\bv\right|\le \kappa_{\max}+\frac{\kappa_{\min}}{2}.
\end{align*}
Therefore, the probability on the left-hand side of the statement is bounded below by $\Pr(Q_{N})$.
Using the estimate $\kappa_{\min}\ge (1/2)T^{2}\exp\left(-2T\|\bfA_{0}\|_{\op}\right)$, we obtain the desired conclusion.
\end{proof}

\subsection{Tail bounds for the martingale term}

Let us recall a matrix-valued martingale.
\begin{equation*}
     \bfM_{N}(t):=\frac{1}{N}\sum_{i=1}^{N}\int_{0}^{t}\diff \bw_{i}(s)(\bx_{i}(s))^{\top}.
\end{equation*}
Furthermore, we define the following Gaussian width and radius of a set $\calD\subset\R^{d\times d}$:
\begin{equation*}
    w(\calD)=\E\left[\sup_{\bfB\in\calD}\langle\textnormal{vec}(\bfB),\bz\rangle\right],\quad \rad(\calD)=\sup_{\bfB\in\calD}\|\bfB\|_{2},
\end{equation*}
where $\bz\sim\calN(\zero,\bfI_{d^{2}})$.

\begin{lemma}[Lemma 3.2 of \citealp{dexheimer2024lasso}]
    There exists a universal constant $c_{0}>0$ such that for any $\calD\subset\R^{d\times d}$ and $u>0$,
    \begin{equation*}
        \Pr\left(\sup_{\bfB\in\calD}\langle \bfM_{N}(T),\bfB\rangle \ind_{Q_{N}}\le c_{0}\sqrt{\frac{\kappa^{\ast}}{N}}(w(\calD)+u\rad(\calD))\right)\ge 1-2\exp(-u^{2}),
    \end{equation*}
    where $\kappa^{\ast}=\kappa_{\max}+\kappa_{\min}/2$.
\end{lemma}

\begin{proof}
    For any $\bfB\in\R^{d\times d}$,
    \begin{align*}
        \frac{1}{N}\sum_{i=1}^{N}\int_{0}^{T}\|\bfB\bx_{i}(t)\|_{2}^{2}\diff t
        &=\frac{1}{N}\sum_{i=1}^{N}\int_{0}^{T}\tr(\bfB\bx_{i}(t)\bx_{i}(t)^{\top}\bfB)\diff t\\
        &=\tr(\bfB\hat{\bfC}_{N}\bfB^{\top})=\sum_{i=1}^{d}\bb_{i}^{\top}\hat{\bfC}_{N}\bb_{i}\le \|\hat{\bfC}_{N}\|_{\op}\|\bfB\|_{2}^{2},
    \end{align*}
    where $\bb_{i}$ denotes the transpose of the $i$-th row vector of $\bfB$ (that is, the $i$-th column vector of $\bfB^{\top}$).
    Let $\bfB_{1},\bfB_{2}\in\calD$ with $\bfB_{1}\neq\bfB_{2}$. Then Bernstein's inequality for continuous martingales derives the following exponential moment bound:
    \begin{align*}
        &\E\left[\exp\left(\frac{\langle\bfM_{N}(T),\bfB_{1}-\bfB_{2}\rangle^{2}\ind_{Q_{N}}}{6N^{-1}\kappa^{\ast}\|\bfB_{1}-\bfB_{2}\|_{2}^{2}}\right)\right]\\
        &= \int_{0}^{\infty}\Pr\left(\exp\left(\frac{\langle\bfM_{N}(T),\bfB_{1}-\bfB_{2}\rangle^{2}\ind_{Q_{N}}}{6N^{-1}\kappa^{\ast}\|\bfB_{1}-\bfB_{2}\|_{2}^{2}}\right)>u\right)\diff u\\
        &\le 1+\int_{1}^{\infty}\Pr\left(|\langle\bfM_{N}(T),\bfB_{1}-\bfB_{2}\rangle|\ind_{Q_{N}}>\sqrt{6N^{-1}\kappa^{\ast}\log u}\|\bfB_{1}-\bfB_{2}\|_{2}\right)\diff u\\
        &\le 1+\int_{1}^{\infty}\Pr\left(\left\{\left|\frac{1}{N}\sum_{i=1}^{N}\int_{0}^{T} (\bx_{i}(s))^{\top}(\bfB_{1}-\bfB_{2})^{\top}\diff\bw_{i}(s)\right|>\frac{\sqrt{6\kappa^{\ast}\log u}\|\bfB_{1}-\bfB_{2}\|_{2}}{\sqrt{N}}\right\}\!\cap\!Q_{N}\right)\diff u\\
        &\le 2.
    \end{align*}
    Therefore, $\{\langle \bfM_{N}(T),\bfB\rangle \ind_{Q_{T}}\rangle\colon \bfB\in\calD\}$ is a subgaussian process such that for any $\bfB_{1},\bfB\in\calD$,
    \begin{equation*}
        \left\|\langle \bfM_{N}(T),\bfB_{1}\rangle \ind_{Q_{T}}-\langle \bfM_{N}(T),\bfB_{2}\rangle \ind_{Q_{T}}\right\|\le \frac{6}{N}\kappa^{\ast}\|\bfB_{1}-\bfB_{2}\|_{2}^{2}.
    \end{equation*}
    Hence, Exercise 8.6.5 of \citet{vershynin2018high} yields the conclusion.
\end{proof}

\begin{proposition}[Proposition 3.3 of \citealp{dexheimer2024lasso}]\label{prop:martingale}
    For $0<\epsilon_{0}<1$, define a norm
    \begin{equation*}
        \|\bfB\|_{\Slope}:=\|\bfB\|_{\ast}\vee \sqrt{\log(4\epsilon_{0}^{-1})}\|\bfB\|_{2},~^{\forall}\bfB\in\R^{d\times d}.
    \end{equation*}
    There exists a universal constant $c_{\Lasso}$ such that for all $N>0$,
    \begin{equation*}
        \Pr\left(\sup_{\bfB\in\R^{d\times d},\bfB\neq\bfO}\frac{\langle\bfM_{N}(T),\bfB\rangle}{\|\bfB\|_{\Slope}}\ind_{Q_{N}}\le c_{\Lasso}\sqrt{\frac{\kappa^{\ast}}{N}}\right)\ge 1-\frac{\epsilon_{0}}{2}.
    \end{equation*}
\end{proposition}

\section{Proofs of main results}\label{append:proofs}

\subsection{Proofs of the upper bounds}\label{append:proofs:upper}
We first give an i.i.d.~version of Lemma 3 of \citet{gaiffas2019sparse} and its extension to the Slope estimator.
Although the proof is essentially identical to the original version, we present it to see that it can be extended to i.i.d.~settings and the Slope estimator.

\begin{lemma}\label{lem:upper:GM19}
    For any matrix $\bfA\in\R^{d\times d}$ and any $\lambda>0$, we have
    \begin{align*}
        &\frac{1}{N}\sum_{i=1}^{N}\left\|\left(\hat{\bfA}_{\Lasso}-\bfA_{0}\right)\bx_{i}\right\|_{L^2}^{2}-\frac{1}{N}\sum_{i=1}^{N}\left\|\left(\bfA-\bfA_{0}\right)\bx_{i}\right\|_{L^2}^{2}\\
        &\le 2\left\langle\bfM_{N}(T),\bfA-\hat{\bfA}_{\Lasso}\right\rangle-\frac{1}{N}\sum_{i=1}^{N}\left\|\left(\bfA-\hat{\bfA}_{\Lasso}\right)\bx_{i}\right\|_{L^2}^{2}+2\lambda\left(\left\|\bfA\right\|_{1}-\left\|\hat{\bfA}_{\Lasso}\right\|_{1}\right),
    \end{align*}
    and
    \begin{align*}
        &\frac{1}{N}\sum_{i=1}^{N}\left\|\left(\hat{\bfA}_{\Slope}-\bfA_{0}\right)\bx_{i}\right\|_{L^2}^{2}-\frac{1}{N}\sum_{i=1}^{N}\left\|\left(\bfA-\bfA_{0}\right)\bx_{i}\right\|_{L^2}^{2}\\
        &\le 2\left\langle\bfM_{N}(T),\bfA-\hat{\bfA}_{\Slope}\right\rangle-\frac{1}{N}\sum_{i=1}^{N}\left\|\left(\bfA-\hat{\bfA}_{\Slope}\right)\bx_{i}\right\|_{L^2}^{2}+2\lambda\left(\left\|\bfA\right\|_{\ast}-\left\|\hat{\bfA}_{\Slope}\right\|_{\ast}\right).
    \end{align*}
\end{lemma}

\begin{proof}
    We obtain that the gradient of $\Loss_{N}(\bfA)=\langle \bfA,\bfM_{N}(T)\rangle-(\bfA-\bfA_{0})\hat{\bfC}_{N}(\bfA-\bfA_{0})^{\top}/2+\bfA_{0}\hat{\bfC}_{N}\bfA_{0}/2$ is $\bfM_{N}(T)+(\bfA-\bfA_{0})\hat{\bfC}_{N}$.
    The optimality of $\hat{\bfA}_{\Lasso}$ (resp.~$\hat{\bfA}_{\Slope}$) derives that there exists $\bfG_{\Lasso}$ (resp.~$\bfG_{\Slope}$) in the subgradient of $\|\cdot\|_{1}$ (resp.~$\|\cdot\|_{\ast}$) at $\hat{\bfA}_{\Lasso}$ (resp.~$\hat{\bfA}_{\Slope}$) satisfying $\bfM_{N}(T)+(\hat{\bfA}_{\Lasso}-\bfA_{0})\hat{\bfC}_{N}+\lambda\bfG_{\Lasso}=\bfO$ (resp.~$\bfM_{N}(T)+(\hat{\bfA}_{\Slope}-\bfA_{0})\hat{\bfC}_{N}+\lambda\bfG_{\Slope}=\bfO$).
    By the definition of subgradients, for any $\bfA\in\R^{d\times d}$,
    \begin{equation*}
        \langle\bfG_{\Lasso},\bfA-\hat{\bfA}_{\Lasso}\rangle\le \|\bfA\|_{1}-\|\hat{\bfA}_{\Lasso}\|_{1},\ \langle\bfG_{\Slope},\bfA-\hat{\bfA}_{\Slope}\rangle\le \|\bfA\|_{\ast}-\|\hat{\bfA}_{\Slope}\|_{\ast}.
    \end{equation*}
    Since $(1/N)\sum_{i=1}^{N}\|\bfM\bx_{i}\|_{L^{2}}^{2}=\langle \bfM^{\top}\bfM,\hat{\bfC}_{N}\rangle$ for any matrix $\bfM$, we derive
    \begin{align*}
        &\frac{1}{N}\sum_{i=1}^{N}\left\|\left(\hat{\bfA}_{\Lasso}-\bfA_{0}\right)\bx_{i}\right\|_{L^2}^{2}-\frac{1}{N}\sum_{i=1}^{N}\left\|\left(\bfA-\bfA_{0}\right)\bx_{i}\right\|_{L^2}^{2}+\frac{1}{N}\sum_{i=1}^{N}\left\|\left(\bfA-\hat{\bfA}_{\Lasso}\right)\bx_{i}\right\|_{L^2}^{2}\\
        &=\left\langle\hat{\bfC}_{N},\left(\hat{\bfA}_{\Lasso}-\bfA_{0}\right)^{\top}\left(\hat{\bfA}_{\Lasso}-\bfA_{0}\right)-\left(\bfA-\bfA_{0}\right)^{\top}\left(\bfA-\bfA_{0}\right)+\left(\bfA-\hat{\bfA}_{\Lasso}\right)^{\top}\left(\bfA-\hat{\bfA}_{\Lasso}\right)\right\rangle\\
        &=\left\langle\hat{\bfC}_{N},
        \hat{\bfA}_{\Lasso}^{\top}\hat{\bfA}_{\Lasso}-2\bfA_{0}^{\top}\hat{\bfA}_{\Lasso}+\bfA_{0}^{\top}\bfA
        -\bfA^{\top}\bfA+2\bfA_{0}^{\top}\bfA-\bfA_{0}^{\top}\bfA_{0}
        +\bfA^{\top}\bfA-2\hat{\bfA}_{\Lasso}^{\top}\bfA+\hat{\bfA}_{\Lasso}^{\top}\hat{\bfA}_{\Lasso}\right\rangle\\
        &=2\left\langle\hat{\bfC}_{N},
        \left(\bfA_{0}-\hat{\bfA}_{\Lasso}\right)^{\top}\left(\bfA-\hat{\bfA}_{\Lasso}\right)\right\rangle= 2\left\langle\bfM_{N}(T)+\lambda\bfG_{\Lasso},\bfA-\hat{\bfA}_{\Lasso}\right\rangle\\
        &\le  2\left\langle\bfM_{N}(T),\bfA-\hat{\bfA}_{\Lasso}\right\rangle+2\lambda\left(\left\|\bfA\right\|_{1}-\left\|\hat{\bfA}_{\Lasso}\right\|_{1}\right),
    \end{align*}
    where we used the symmetricity of $\hat{C}_{N}$ to obtain the second identity.
    The same argument holds for the Slope estimator clearly.
\end{proof}

We first give the oracle inequality for the Lasso estimation.
\begin{proof}[Proof of Proposition \ref{prop:upper:lasso}]
    Note that for any $\bfB\in\R^{d\times d}$,
    \begin{equation*}
        \frac{1}{N}\sum_{i=1}^{N}\int_{0}^{T}\|\bfB\bx_{i}(t)\|_{2}^{2}\diff t
        =\frac{1}{N}\sum_{i=1}^{N}\int_{0}^{T}\tr(\bfB\bx_{i}(t)\bx_{i}(t)^{\top}\bfB)\diff t
        =\tr(\bfB\hat{\bfC}_{N}\bfB^{\top})=\sum_{i=1}^{d}\bb_{i}^{\top}\hat{\bfC}_{N}\bb_{i},
    \end{equation*}
    where $\bb_{i}$ denotes the transpose of the $i$-th row vector of $\bfB$ (or equivalently, the $i$-th column vector of $\bfB^{\top}$).
    Therefore, we obtain
    \begin{equation*}
        \inf_{\bfB\in\R^{d\times d}\backslash\{\bfO\}}\frac{\frac{1}{N}\sum_{i=1}^{N}\|\bfB\bx_{i}\|_{L^{2}}^{2}}{\|\bfB\|_{2}^{2}}=\lambda_{\min}(\hat{\bfC}_{N}).
    \end{equation*}
    We suppose that both of the following inequalities hold:
    \begin{align}
        &\inf_{\bfB\in\R^{d\times d}\backslash\{\bfO\}}\frac{\frac{1}{N}\sum_{i=1}^{N}\|\bfB\bx_{i}\|_{L^{2}}^{2}}{\|\bfB\|_{2}^{2}}\ge \frac{\kappa_{\min}}{2},\label{eq:lasso:upper:regular-covariance}\\
        &\sup_{\bfB\in\R^{d\times d}\backslash\{\bfO\}}\frac{\langle \bfM_{N}(T),\bfB\rangle}{\|\bfB\|_{\Slope}}\le c_{\Lasso}\sqrt{\frac{\kappa^{\ast}}{N}}\label{eq:lasso:upper:bounded-martingale},
    \end{align}
    where $\|\cdot\|_{\Slope}:=\|\cdot\|_{\ast}\vee\sqrt{\log(4\epsilon_{0}^{-1})}\|\cdot\|_{2}$.
    By Propositions \ref{prop:spectra} and \ref{prop:martingale}, the probability of this event is bounded below by $1-\epsilon_{0}/2-8\times9^{d}\times\exp(-N/(c(1+d/N)))$.
    Lemma \ref{lem:upper:GM19} leads to
    \begin{align*}
        &\frac{1}{N}\sum_{i=1}^{N}\left\|\left(\hat{\bfA}_{\Lasso}-\bfA_{0}\right)\bx_{i}\right\|_{L^2}^{2}-\frac{1}{N}\sum_{i=1}^{N}\left\|\left(\bfA-\bfA_{0}\right)\bx_{i}\right\|_{L^2}^{2}\\
        &\le 2\left\langle\bfM_{N}(T),\bfB\right\rangle-\frac{1}{N}\sum_{i=1}^{N}\left\|\bfB\bx_{i}\right\|_{L^2}^{2}+2\lambda\left(\left\|\bfA\right\|_{1}-\left\|\hat{\bfA}_{\Lasso}\right\|_{1}\right),
    \end{align*}
    where $\bfB:=\bfA-\hat{\bfA}_{\Lasso}$.
    Letting $\Delta_{\Lasso}:=\lambda_{\Lasso}\|\bfB\|_{1}+2\langle\bfM_{N}(T),\bfB\rangle+2\lambda_{\Lasso}(\|\bfA\|_{1}-\|\hat{\bfA}_{\Lasso}\|_{1})$ and adding $\lambda_{\Lasso}\|\bfB\|_{1}$ to both the sides, we obtain
    \begin{equation}\label{eq:lasso:upper:objective}
        \frac{1}{N}\sum_{i=1}^{N}\left\|\left(\hat{\bfA}_{\Lasso}-\bfA_{0}\right)\bx_{i}\right\|_{L^2}^{2}+\lambda_{\Lasso}\|\bfB\|_{1}
        \le \frac{1}{N}\sum_{i=1}^{N}\left\|\left(\bfA-\bfA_{0}\right)\bx_{i}\right\|_{L^2}^{2}-\frac{1}{N}\sum_{i=1}^{N}\left\|\bfB\bx_{i}\right\|_{L^2}^{2}+\Delta_{\Lasso}.
    \end{equation}
    Regarding $\|\cdot\|_{\ast}$ appearing in the definition of $\|\cdot\|_{\Slope}$, we have
    \begin{align*}
        \|\bfB\|_{\ast}&\le \|\bfB\|_{2}\sqrt{\sum_{i=1}^{s}\log(2d^{2}/i)}+\sum_{i=s+1}^{d^{2}}(\vect(\bfB)^{(i)})^{\sharp}\sqrt{\log(2d^{2}/i)}\\
        &\le \sqrt{\log (2ed^{2}/s)}\left(\sqrt{s}\|\bfB\|_{2}+\sum_{i=s+1}^{d^{2}}(\vect(\bfB)^{(i)})^{\sharp}\right)=:F_{\Lasso}(\bfB),
    \end{align*}
    where the first inequality is given by the Cauchy--Schwarz inequality, and the second inequality is owing to Equation~(2.7) of \citet{bellec2018slope}.
    We obtain 
    \begin{align*}
        \Delta_{\Lasso}
        &=2\lambda_{\Lasso}\left(\frac{1}{2}\left\|\bfB\right\|_{1}+\left\|\bfA\right\|_{1}-\left\|\hat{\bfA}_{\Lasso}\right\|_{1}\right)+2\left\langle\bfM_{N}(T),\bfB\right\rangle\\
        &\le 2\lambda_{\Lasso}\left(\frac{1}{2}\left\|\bfB\right\|_{1}+\left\|\bfA\right\|_{1}-\left\|\hat{\bfA}_{\Lasso}\right\|_{1}\right)+\frac{1}{\sqrt{\log(2ed^{2}/s)}}\lambda_{\Lasso}\left(F_{\Lasso}(\bfB)\vee\sqrt{ \log(4\epsilon_{0}^{-1})}\|\bfB\|_{2}\right)\\
        &\le \lambda_{\Lasso}\left(3\sqrt{s}\|\bfB\|_{2}-\sum_{i=s+1}^{d^{2}}(\vect(\bfB)^{(i)})^{\sharp}\right.\\
        &\qquad\qquad\left.+\frac{1}{\sqrt{\log(2ed^{2}/s)}}\left(F_{\Lasso}(\bfB)\vee\sqrt{ \frac{2\log(4\epsilon_{0}^{-1})}{\kappa_{\min}}\left(\frac{1}{N}\sum_{i=1}^{N}\|\bfB\bx_{i}\|_{L^{2}}^{2}\right)}\right)\right),
    \end{align*}
    where the second line is owing to Eq.~\eqref{eq:lasso:upper:bounded-martingale} and $\lambda_{\Lasso}(\log(2ed^{2}/s))^{-1/2}\ge 2c_{\Lasso}\sqrt{N^{-1}\kappa^{\ast}}$, and 
    the last line is given by Lemma A.1 of \citet{bellec2018slope} (with $\tau=1/2$ under their notation and the condition $\|\bfA\|_{0}\le s$) and Eq.~\eqref{eq:lasso:upper:regular-covariance}.
    We first examine the case $F_{\Lasso}(\bfB)\le\sqrt{ (2\log(4\epsilon_{0}^{-1})/\kappa_{\min})(\frac{1}{N}\sum_{i=1}^{N}\|\bfB\bx_{i}\|_{L^{2}}^{2})}$.
    The non-negativity of $(\vect(\bfB)^{(i)})^{\sharp}$ and the bound on $F_{\Lasso}(\bfB)$ together imply
    \begin{equation*}
        \|\bfB\|_{2}\le \sqrt{\frac{2\log(4\epsilon_{0}^{-1})}{s\log(2ed^{2}/s)\kappa_{\min}}\left(\frac{1}{N}\sum_{i=1}^{N}\|\bfB\bx_{i}\|_{L^{2}}^{2}\right)}.
    \end{equation*}
    Therefore, 
    \begin{align*}
        \Delta_{\Lasso}&\le 3\lambda_{\Lasso}\sqrt{s}\|\bfB\|_{2}+\lambda_{\Lasso}\sqrt{\frac{2\log(4\epsilon_{0}^{-1})}{s\log(2ed^{2}/s)\kappa_{\min}}\left(\frac{1}{N}\sum_{i=1}^{N}\|\bfB\bx_{i}\|_{L^{2}}^{2}\right)}\\
        &\le 4\lambda_{\Lasso}\sqrt{\frac{2\log(4\epsilon_{0}^{-1})}{s\log(2ed^{2}/s)\kappa_{\min}}\left(\frac{1}{N}\sum_{i=1}^{N}\|\bfB\bx_{i}\|_{L^{2}}^{2}\right)}\\
        &\le 8\lambda_{\Lasso}^{2}\frac{\log(4\epsilon_{0}^{-1})}{s\log(2ed^{2}/s)\kappa_{\min}}+\frac{1}{N}\sum_{i=1}^{N}\|\bfB\bx_{i}\|_{L^{2}}^{2},
    \end{align*}
    where the first inequality uses the non-negativity of $(\vect(\bfB)^{(i)})^{\sharp}$, and the last inequality is given by the bound $ab\le a^{2}/4+b^{2}$ for any $a,b\ge 0$.
    For the other case $F_{\Lasso}(\bfB)>\sqrt{ (2\log(4\epsilon_{0}^{-1})/\kappa_{\min})(\frac{1}{N}\sum_{i=1}^{N}\|\bfB\bx_{i}\|_{L^{2}}^{2})}$, 
    \begin{equation*}
        \Delta_{\Lasso}\le 4\lambda_{\Lasso}\sqrt{s}\left\|\bfB\right\|_{2}\le 4\lambda_{\Lasso}\sqrt{\frac{2s}{\kappa_{\min}}\left(\frac{1}{N}\sum_{i=1}^{N}\|\bfB\bx_{i}\|_{L^{2}}^{2}\right)}\le 8\lambda_{\Lasso}^{2}\frac{s}{\kappa_{\min}}+\frac{1}{N}\sum_{i=1}^{N}\|\bfB\bx_{i}\|_{L^{2}}^{2}.
    \end{equation*}
    Therefore, these estimates on $\Delta_{\Lasso}$ and Eq.~\eqref{eq:lasso:upper:objective} derive the statement.
\end{proof}

\begin{proof}[Proof of Corollary \ref{cor:upper:lasso}]
    The bounds on the $L^{2}$- and $\ell^{1}$-distance are immediate.
    The bound on the $\ell^{2}$-distance also holds since we consider on the event where Eq.~\eqref{eq:lasso:upper:regular-covariance} holds.
\end{proof}

We next exhibit the oracle inequality for the Slope estimation.

\begin{proof}[Proof of Proposition \ref{prop:upper:slope}]
    We again suppose that the inequalities \eqref{eq:lasso:upper:regular-covariance} and \eqref{eq:lasso:upper:bounded-martingale} hold true.
    The probability of the event that both of them hold true is bounded below by $1-\epsilon_{0}/2-8\times9^{d}\times\exp(-N/(c(1+d/N)))$.

    We set $\bfB:=\bfA-\hat{\bfA}_{\Slope}$ and $\Delta_{\Slope}=2\langle \bfM_{N}(T),\bfB\rangle+4c_{\Slope}\sqrt{N^{-1}}(\|\bfA\|_{\ast}-\|\hat{\bfA}\|_{\ast}+(1/2)\|\bfB\|_{\ast})$.
    Lemma \ref{lem:upper:GM19} yields that 
    \begin{equation}\label{eq:slope:upper:objective}
        \frac{1}{N}\sum_{i=1}^{N}\left\|\left(\hat{\bfA}_{\Slope}-\bfA_{0}\right)\bx_{i}\right\|_{L^2}^{2}+\frac{2c_{\Slope}}{\sqrt{N}}\|\bfB\|_{\ast}
        \le \frac{1}{N}\sum_{i=1}^{N}\left\|\left(\bfA-\bfA_{0}\right)\bx_{i}\right\|_{L^2}^{2}-\frac{1}{N}\sum_{i=1}^{N}\left\|\bfB\bx_{i}\right\|_{L^2}^{2}+\Delta_{\Slope}.
    \end{equation}
    We have
    \begin{align*}
        \Delta_{\Slope}
        &\le \frac{4c_{\Slope}}{\sqrt{N}}\left(\frac{1}{2}\|\bfB\|_{\ast}+\|\bfA\|_{\ast}-\|\hat{\bfA}_{\Slope}\|_{\ast}\right)+2c_{\Lasso}\sqrt{\frac{\kappa^{\ast}}{N}}\|\bfB\|_{\Slope}\\
        &\le  \frac{4c_{\Slope}}{\sqrt{N}}\left(\frac{1}{2}\|\bfB\|_{\Slope}+\frac{1}{2}\|\bfB\|_{\ast}+\|\bfA\|_{\ast}-\|\hat{\bfA}_{\Slope}\|_{\ast}\right)\\
        &\le \frac{4c_{\Slope}}{\sqrt{N}}\left(\frac{1}{2}\|\bfB\|_{\Slope}+\frac{3}{2}\sqrt{\sum_{i=1}^{s}\log\left(\frac{2d^{2}}{i}\right)}\|\bfB\|_{2}-\frac{1}{2}\sum_{i=s+1}^{d^{2}}(\vect(\bfB)^{(i)})^{\ast}\sqrt{\log\left(\frac{2d^{2}}{i}\right)}\right)\\
        &\le \frac{4c_{\Slope}}{\sqrt{N}}\left(\frac{1}{2}\|\bfB\|_{\Slope}+\frac{3}{2}\sqrt{s\log\left(\frac{2ed^{2}}{s}\right)}\|\bfB\|_{2}-\frac{1}{2}\sum_{i=s+1}^{d^{2}}(\vect(\bfB)^{(i)})^{\ast}\sqrt{\log\left(\frac{2d^{2}}{i}\right)}\right),
    \end{align*}
    where the first line uses \eqref{eq:lasso:upper:bounded-martingale}, the third line is derived by Lemma A.1 of \citet{bellec2018slope} (with $\tau=1/2$ under their notation), and the last line is given by Eq.~(2.7) of \citet{bellec2018slope}.
    We also have
    \begin{align*}
        \|\bfB\|_{\ast}&\le \|\bfB\|_{2}\sqrt{\sum_{i=1}^{s}\log(2d^{2}/i)}+\sum_{i=s+1}^{d^{2}}(\vect(\bfB)^{(i)})^{\sharp}\sqrt{\log(2d^{2}/i)}\\
        &\le \sqrt{s\log (2ed^{2}/s)}\|\bfB\|_{2}+\sum_{i=s+1}^{d^{2}}(\vect(\bfB)^{(i)})^{\sharp}\sqrt{\log(2d^{2}/i)}=:F_{\Slope}(\bfB),
    \end{align*}
    where the last inequality uses Eq.~(2.7) of \citet{bellec2018slope}.
    Therefore, as the proof of Proposition \ref{prop:upper:lasso}, we derive
    \begin{align*}
        \Delta_{\Slope}\le \frac{4c_{\Slope}}{\sqrt{N}}&\left(\frac{1}{2}\left(F_{\Slope}(\bfB)\vee \sqrt{\frac{2\log(4\epsilon_{0}^{-1})}{\kappa_{\min}}\left(\frac{1}{N}\sum_{i=1}^{N}\|\bfB\bx_{i}\|_{L^{2}}^{2}\right)}\right)\right.\\
        &\quad\left.+\frac{3}{2}\sqrt{s\log\left(\frac{2ed^{2}}{s}\right)}\|\bfB\|_{2}-\frac{1}{2}\sum_{i=s+1}^{d^{2}}(\vect(\bfB)^{(i)})^{\ast}\sqrt{\log\left(\frac{2d^{2}}{i}\right)}\right).
    \end{align*}
    We first consider the case $F_{\Slope}(\bfB)\le \sqrt{(2\log(4\epsilon_{0}^{-1})/\kappa_{\min})(\frac{1}{N}\sum_{i=1}^{N}\|\bfB\bx_{i}\|_{L^{2}}^{2})}$. It implies
    \begin{equation*}
        \|\bfB\|_{2}\le \sqrt{\frac{2\log(4\epsilon_{0}^{-1})}{s\log(2ed^{2}/s)\kappa_{\min}}\left(\frac{1}{N}\sum_{i=1}^{N}\|\bfB\bx_{i}\|_{L^{2}}^{2}\right)},
    \end{equation*}
    and thus
    \begin{align*}
        \Delta_{\Slope}&\le \frac{4c_{\Slope}}{\sqrt{N}}\left(\frac{1}{2}\sqrt{\frac{2\log(4\epsilon_{0}^{-1})}{\kappa_{\min}}\left(\frac{1}{N}\sum_{i=1}^{N}\|\bfB\bx_{i}\|_{L^{2}}^{2}\right)}+\frac{3}{2}\sqrt{s\log\left(\frac{2ed^{2}}{s}\right)}\|\bfB\|_{2}\right)\\
        &\le \frac{8c_{\Slope}}{\sqrt{N}}\sqrt{\frac{2\log(4\epsilon_{0}^{-1})}{\kappa_{\min}}\left(\frac{1}{N}\sum_{i=1}^{N}\|\bfB\bx_{i}\|_{L^{2}}^{2}\right)}
        \le \frac{32c_{\Slope}^{2}\log(4\epsilon_{0}^{-1})}{N\kappa_{\min}}+\frac{1}{N}\sum_{i=1}^{N}\|\bfB\bx_{i}\|_{L^{2}}^{2}.
    \end{align*}
    Under the case $F_{\Slope}(\bfB)> \sqrt{(2\log(4\epsilon_{0}^{-1})/\kappa_{\min})(\frac{1}{N}\sum_{i=1}^{N}\|\bfB\bx_{i}\|_{L^{2}}^{2})}$, since we suppose that \eqref{eq:lasso:upper:regular-covariance} holds true, we derive
    \begin{align*}
        \Delta_{\Slope}&\le \frac{4c_{\Slope}}{\sqrt{N}}\left(\frac{1}{2}F_{\Slope}(\bfB)+\frac{3}{2}\sqrt{s\log\left(\frac{2ed^{2}}{s}\right)}\|\bfB\|_{2}-\frac{1}{2}\sum_{i=s+1}^{d^{2}}(\vect(\bfB)^{(i)})^{\ast}\sqrt{\log\left(\frac{2d^{2}}{i}\right)}\right)\\
        &\le 8c_{\Slope}\sqrt{\frac{s\log(2ed^{2}/s)}{N}}\|\bfB\|_{2}\le 8c_{\Slope}\sqrt{\frac{2s\log(2ed^{2}/s)}{N\kappa_{\min}}\left(\frac{1}{N}\sum_{i=1}^{N}\|\bfB\bx_{i}\|_{L^{2}}^{2}\right)}\\
        &\le \frac{32c_{\Slope}^{2}s\log(2ed^{2}/s)}{N\kappa_{\min}}+\frac{1}{N}\sum_{i=1}^{N}\|\bfB\bx_{i}\|_{L^{2}}^{2}.
    \end{align*}
    Therefore, these estimates on $\Delta_{\Slope}$ combined with Eq.~\eqref{eq:slope:upper:objective} derive the statement.
\end{proof}

\begin{proof}[Proof of Corollarty \ref{cor:upper:slope}]
    Since $\|\cdot\|_{\ast}\ge \log 2\|\cdot\|_{1}$,
    the proof is immediate by adapting the proof of Corollary \ref{cor:upper:lasso}. 
\end{proof}

\subsection{Proof of the lower bounds}\label{append:proofs:lower}
We follow the line of \citet{dexheimer2024lasso}.
Note that our limiting information matrix is not the identity matrix up to multiplicative constants, but it suffices to show that the eigenvalues are bounded over $s,d,\epsilon$.
In this section, we let $D(\cdot\|\cdot)$ denote the Kullback--Leibler divergence.

\begin{lemma}[\citealp{gaiffas2019sparse}]\label{lem:gm19}
    Let $\bfA=-(\alpha\bfI_{d}+\bfB)$ for a positive number $\alpha>0$ and an antisymmetric matrix $\bfB$.
    Then, for all $T\in(0,\infty]$,
    \begin{align*}
        \int_{0}^{T}\exp\left(\bfA (T-t)\right)\exp\left(\bfA^{\top}(T-t)\right)\diff t=\frac{1-\exp(-2\alpha T)}{2\alpha}\bfI_{d}.
    \end{align*}
\end{lemma}

\begin{proof}
    The proof is almost identical to that of Lemma 6 of \citet{gaiffas2019sparse}. 
    As $\bfB^{\top}=-\bfB$, $i\bfB$ is Hermitian and thus unitarily diagonalizable, that is, $\bfB=i\bfU\bfD\bfU^{\ast}$ for a unitary matrix $\bfU$ and a real diagonal matrix $\bfD$.
    Then, $\exp(-(\alpha\bfI+\bfB)t)=\exp(-\alpha t)\bfU\exp(-i\bfD t)\bfU^{\ast}$ and $\exp(-(\alpha\bfI+\bfB)^{\top}t)=\exp(-\alpha t)\bfU\exp(i\bfD t)\bfU^{\ast}$.
    We obtain that
    \begin{equation*}
        \int_{0}^{T}\exp\left(\bfA (T-t)\right)\exp\left(\bfA^{\top}(T-t)\right)\diff t=\int_{0}^{T}\exp\left(-2\alpha (T-t)\right)\diff t\bfI_{d}=\frac{1-\exp(-2\alpha T)}{2\alpha}\bfI_{d}.
    \end{equation*}
    This is the desired conclusion.
\end{proof}
Let $\bbP_{\bfA}^{N}$ denote the law of $\{\bx_{i}(t)\colon i\in\{1,\ldots,N\},t\in[0,1]\}$ whose true value of the drift parameter is $\bfA$.

\begin{proof}[Proof of Proposition \ref{prop:lower}]
We fix $r$, the largest even number with $r\le (s-d)/2$ and let $\Omega_{r}$ be the set al all antisymmetric matrices in $\{-1,0,1\}^{d\times d}$ whose number of nonzero elements is exactly $r$.
Due to Lemma F.1 of \citet{bellec2018slope} and Theorem 2.7 of \citet{dexheimer2024lasso}, there exists a set $\tilde{\Omega}_{r}\subset\Omega_{r}$, for all $\bfB,\bfB'\in\tilde{\Omega}_{r}$ with $\bfB\neq\bfB'$,
\begin{align}
    &\log\left(\left|\tilde{\Omega}_{r}\right|\right)\ge cr\log(ed(d-1)/r)\ge cr\log(ed^{2}/s),\label{eq:B:1}\\
    &\left\|\bfB-\bfB'\right\|_{p}^{p}\ge r/8\ge ((s-d)/2-1)/8\ge s/64,\label{eq:B:2}
\end{align}
where $c>0$ is a universal constant.
For arbitrary $w>0$, we set
\begin{align*}
    \Omega_{w}:=\left\{-\frac{1}{2}\bfI_{d}-w\bfB\colon \bfB\in\tilde{\Omega}_{r}\right\}.
\end{align*}
For any $\bfA,\bfA'\in\Omega_{w}$, Lemma \ref{lem:gm19} yields that
\begin{align*}
    &D\left(\bbP_{\bfA}^{N}\|\bbP_{\bfA'}^{N}\right)=\E_{\bfA}\left[\log\frac{\bbP_{\bfA}^{N}}{\bbP_{\bfA'}^{N}}(\bx)\right]=\E_{\bfA}\left[\frac{1}{2}\sum_{i=1}^{N}\int_{0}^{1}\left\|\left(\bfA-\bfA'\right)\bx_{i}(t)\right\|_{2}^{2}\diff t\right]\\
    &=\frac{1}{2}\sum_{i=1}^{N}\int_{0}^{1}\E\left[\left\|\left(\bfA-\bfA'\right)\int_{0}^{t}\exp((t-s)\bfA)\diff \bw_{i}(s)\right\|_{2}^{2}\right]\diff t\\
    &=\frac{1}{2}\sum_{i=1}^{N}\int_{0}^{1}\tr\left(\left(\bfA-\bfA'\right)\E\left[\left(\int_{0}^{t}\exp((t-s)\bfA)\diff \bw_{i}(s)\right)^{\otimes 2}\right]\left(\bfA-\bfA'\right)^{\top}\right)\diff t\\
    &=\frac{1}{2}\sum_{i=1}^{N}\int_{0}^{1}\tr\left(\left(\bfA-\bfA'\right)\int_{0}^{t}\left(\exp((t-s)\bfA)\exp((t-s)\bfA)^{\top}\diff s\right)\left(\bfA-\bfA'\right)^{\top}\right)\diff t\\
    &=\frac{1}{2}N\left\|\bfA-\bfA'\right\|_{2}^{2}\int_{0}^{1}\frac{1-e^{-2t}}{2}\diff t\\
    &\le \frac{1}{2}Nw^{2}r,
\end{align*}
and for $\bfA\neq\bfA'$ and $p\ge1$, Eq.~\eqref{eq:B:2} leads to
\begin{equation*}
    \left\|\bfA-\bfA'\right\|_{p}^{p}\ge sw^{p}/64.
\end{equation*}
By letting $w=\sqrt{cN^{-1}\log(ed^{2}/s)/25}$ and setting $\Omega=\Omega_{w}$, we have
\begin{align*}
    D\left(\bbP_{\bfA}^{N}\|\bbP_{\bfA'}^{N}\right)&< \log(|\Omega|)/8,\\
    \|\bfA-\bfA'\|_{p}^{p}&\ge s(cN^{-1}\log(ed^{2}/s))^{p/2}/320.
\end{align*}
Therefore, Theorem 2.7 of \citet{tsybakov2009introduction} completes the proof.
\end{proof}
 
\begin{remark}
    We can generalize the setting $\bxi_{i}=\zero_{d}$ by adding a few lines to the proof. 
    For example, if $\E[\bxi_{i}\bxi_{i}^\top]=\bfI_{d}$, then for any $\bfA,\bfA'\in\Omega_{w}$,
    \begin{align*}
        &D\left(\bbP_{\bfA}^{N}\|\bbP_{\bfA'}^{N}\right)\\
        &=\frac{1}{2}\sum_{i=1}^{N}\int_{0}^{1}\E\left[\left\|\left(\bfA-\bfA'\right)\exp(t\bfA)\bxi_{i}\right\|_{2}^{2}+\left\|\left(\bfA-\bfA'\right)\int_{0}^{t}\exp((t-s)\bfA)\diff \bw_{i}(s)\right\|_{2}^{2}\right]\diff t\\
        &\le \frac{1}{2}\sum_{i=1}^{N}\int_{0}^{1}\exp(-t)\left\|\left(\bfA-\bfA'\right)\right\|_{2}^{2}\diff t+\frac{1}{2}Nw^{2}r\le \frac{3}{2}Nw^{2}r,
\end{align*}
where the first inequality holds by $\exp(\bfA t)\exp(\bfA^\top t)=e^{-t}\bfI_{d}$.
Changing the scale of $w$ accordingly leads to the proof of another lower bound with random initial values.
\end{remark}
 
\section{Oracle inequality under $d\not\ll N$}\label{append:r}
We show an oracle inequality for Lasso under possibly $d\not\ll N$.
Though its rate contains additional logarithmic dependence, it exhibits that we can identify Brownian motions (i.e., $\bfA_{0}=\bfO$).
The proof strategy is mostly adapted from \citet{ciolek2020dantzig}.

\subsection{Notation}
We introduce notation used within this section.
For any $J\subset\{1,\ldots,d_{1}\}\times\{1,\ldots,d_{2}\}$ and $d_{1}\times d_{2}$ matrix $\bfA$, $\bfA_{|J}$ denotes a $d_{1}\times d_{2}$ matrix defined as
\begin{equation*}
    \bfA_{|J}^{(i,j)}=\bfA^{(i,j)}\ind_{J}(i,j).
\end{equation*}
For a $d\times d$ square matrix $\bfA$, $\diag(\bfA)=\bfA_{|\bigcup_{i=1}^{d}\{(i,i)\}}$. 
We define a cone $\calC(s,c_0)\subset\R^{d\times d}$ with $s\in\N$ and $c_{0}>0$ as follows:
\begin{equation*}
    \calC(s,c_0):=\left\{\bfA\in\R^{d\times d}:\|\bfA\|_{1}\le (1+c_{0})\|\bfA_{|\calI_{s}(\bfA)}\|_{1}\right\},
\end{equation*}
where $\calI_{s}(\bfA)$ is the set of indices $(i,j)\in\{1,\ldots,d\}\times\{1,\ldots,d\}$ corresponding to the elements of $\bfA$ whose absolute values are among the $s$ largest ones.
We also define sets $\calK(s)$ and $\calK_{2}(s)$ with $s\in\N$ as
\begin{equation*}
    \calK(s):=\left\{\bfA\in\R^{d\times d}:\|\bfA\|_{0}\le s\right\},\ \calK_{2}(s):=\left\{\bfA\in\R^{d\times d}:\|\bfA\|_{0}\le s,\ \|\bfA\|_{2}\le 1\right\}.
\end{equation*}
Define the effective rank of $\bSigma$ as 
\begin{equation*}
    r(\bSigma)=\begin{cases}
        \frac{\tr(\Sigma)}{\|\Sigma\|_{\op}} & \text{ if }\|\Sigma\|_{\op}>0,\\
        0 & \text{ otherwise}.
    \end{cases}
\end{equation*}
Note that $r(\bSigma)\le d$; it works as an instrinsic dimension of the distribution of $\bxi_{i}$'s.

\subsection{Pointwise concentration}
Though Proposition \ref{prop:pointwise:d} is sufficient under the $d\ll N$ regime, we extend it to a `dimension-free' version by introducing the effective rank $r(\bSigma)$.

The next proposition is a corollary of \citet{koltchinskii2017concentration}; its pre-Gaussian assumption is automatically satisfied as we only consider Euclidean spaces.
\begin{proposition}[corollary of Theorem 9 in \citealp{koltchinskii2017concentration}]\label{prop:koltchinskii2017}
    Suppose that Assumption \ref{assm:initial} holds.
    There exists a universal constant $a>0$ such that
    \begin{equation*}
        \Pr\left(\left\|\frac{1}{N}\sum_{i=1}^{N}\bxi_{i}\bxi_{i}^{\top}-\bSigma\right\|_{\op}\ge a\|\bSigma\|_{\op}\left(\sqrt{\frac{r(\bSigma)+t}{N}}+\frac{r(\bSigma)+t}{N}\right)\right)\le e^{-t}.
    \end{equation*}
\end{proposition}

Using this concentration instead of Eq.~\eqref{eq:tail:initials:d}, we derive the following concentration bound.
\begin{proposition}[pointwise concentration with the effective rank]\label{prop:pointwise:r}
    Under Assumption \ref{assm:initial}, there exists a positive constant $c\ge1$ dependent only on $T$, $\|\bfA_0\|_{\op}$, $K$, and $\|\bSigma\|_{\op}$ such that for all $u>0$,
    \begin{equation*}
        \sup_{\bv\in\bbS^{d-1}}\Pr\left(\left|\bv^{\top}\left(\hat{\bfC}_{N}-\tilde{\bfC}_{N}\right)\bv\right|\ge u\right)\le 2e^{-N}+2\exp\left(-\frac{Nu^{2}}{c(1+r(\bSigma)/N)}\right) + 2\exp\left(-\frac{1}{c}\frac{Nu^{2}}{1+u}\right).
    \end{equation*}
\end{proposition}

\subsection{Restricted eigenvalue condition}
\begin{proposition}[Theorem 3.3 of \citealp{ciolek2020dantzig}]\label{prop:ciolek2020dantzig:thm0303}
    There exists a positive constant $c\ge1$ dependent only on $T,\|\bfA_{0}\|_{\op},K,\|\bSigma\|_{\op}$ such that for any $s\in\N$, $c_{0}>0$, and $\epsilon_{0}\in(0,1)$, if $N\ge r(\bSigma)\vee [c(c_{0}+2)^{2}(2s\log(29d^{2}/s)+\log (6d)+\log(1/\epsilon_{0}))]$, it holds that
    \begin{equation*}
        \Pr\left(\inf_{\bfB\in\calC(s,c_{0})\backslash\{\bfO\}}\frac{\frac{1}{N}\sum_{i=1}^{N}\|\bfB\bx_{i}\|_{L^{2}}^{2}}{\|\bfB\|_{2}^{2}}\ge \frac{\kappa_{\min}}{2}\right)\ge 1-\epsilon_{0}.
    \end{equation*}
\end{proposition}

\begin{proof}
    By the definition of $\hat{\bfC}_{N}$, we have that
    \begin{align*}
        \frac{\frac{1}{N}\sum_{i=1}^{N}\|\bfB\bx_{i}\|_{L^{2}}^{2}}{\|\bfB\|_{2}^{2}}&=\frac{\tr(\bfB\hat{\bfC}_{N}\bfB^\top)}{\|\bfB\|_{2}^{2}}=\frac{\tr(\bfB\bfC_{\infty}\bfB^\top)}{\|\bfB\|_{2}^{2}}-\frac{\tr(\bfB(\bfC_{\infty}-\hat{\bfC}_{N})\bfB^\top)}{\|\bfB\|_{2}^{2}}\\
        &\ge \kappa_{\min}-\frac{|\tr(\bfB(\bfC_{\infty}-\hat{\bfC}_{N})\bfB^\top)|}{\|\bfB\|_{2}^{2}}.
    \end{align*}
    Lemma 6.1 of \citet{ciolek2020dantzig} yields that
    \begin{equation*}
        \sup_{\bfB\in\calC(s,c_{0})\backslash\{\bfO\}}\frac{|\tr(\bfB(\bfC_{\infty}-\hat{\bfC}_{N})\bfB^\top)|}{\|\bfB\|_{2}^{2}}\le 3(c_{0}+2)^{2}\sup_{\bfB\in\calK(2s)\backslash\{\bfO\}}\frac{|\tr(\bfB(\bfC_{\infty}-\hat{\bfC}_{N})\bfB^\top)|}{\|\bfB\|_{2}^{2}}.
    \end{equation*}
    Hence, it suffices to consider a uniform concentration over $\calK(2s)$.
    Considering the representation of the transpose $\bfB^\top=[\bb_{1},\ldots,\bb_{d}]$ with column vectors $\bb_{j}\in\R^{d}$, we derive
    \begin{align*}
        &\frac{|\tr(\bfB(\bfC_{\infty}-\hat{\bfC}_{N})\bfB^\top)|}{\|\bfB\|_{2}^{2}}=\frac{|\sum_{j=1}^{d}\bb_{j}^{\top}(\bfC_{\infty}-\hat{\bfC}_{N})\bb_{j}|}{\sum_{j=1}^{d}\|\bb_{j}\|_{2}^{2}}\le \frac{\sum_{j=1}^{d}|\bb_{j}^{\top}(\bfC_{\infty}-\hat{\bfC}_{N})\bb_{j}|}{\sum_{j=1}^{d}\|\bb_{j}\|_{2}^{2}}\\
        &=\sum_{j=1,\ldots,d:\bb_j\neq\zero}\frac{\|\bb_{j}\|_{2}^{2}}{\sum_{j=1}^{d}\|\bb_{j}\|_{2}^{2}}\frac{|\bb_{j}^{\top}(\bfC_{\infty}-\hat{\bfC}_{N})\bb_{i}|}{\|\bb_{j}\|_{2}^{2}}
        \le \max_{j=1,\ldots,d}\frac{|\bb_{j}^{\top}(\bfC_{\infty}-\hat{\bfC}_{N})\bb_{i}|}{\|\bb_{j}\|_{2}^{2}}.
    \end{align*}
    Therefore, Proposition \ref{prop:pointwise:r} combined with the assumption $N\ge r(\bSigma)$ gives that for any $u>0$,
    \begin{align*}
        &\Pr\left(\frac{|\tr(\bfB(\bfC_{\infty}-\hat{\bfC}_{N})\bfB^\top)|}{\|\bfB\|_{2}^{2}}\ge u\right)\le \Pr\left(\max_{j=1,\ldots,d}\frac{|\bb_{j}^{\top}(\bfC_{\infty}-\hat{\bfC}_{N})\bb_{i}|}{\|\bb_{j}\|_{2}^{2}}\ge u\right)\\
        &\le d\Pr\left(\frac{|\bb_{j}^{\top}(\bfC_{\infty}-\hat{\bfC}_{N})\bb_{i}|}{\|\bb_{j}\|_{2}^{2}}\ge u\right)\le 2d\left(e^{-N}+\exp\left(-\frac{Nu^{2}}{c}\right) + \exp\left(-\frac{1}{c}\frac{Nu^{2}}{1+u}\right)\right)\\
        &\le 2d\left(e^{-N}+2\exp\left(-\frac{1}{c}\frac{Nu^{2}}{1+u}\right)\right).
    \end{align*}
    We simplify the numerical constant using $21e/2\le 29$.
    Lemma 6.2 of \citet{ciolek2020dantzig} yields that
    \begin{equation*}
        \Pr\left(\sup_{\bfB\in\calK(2s)\backslash\{\bfO\}}\frac{|\tr(\bfB(\bfC_{\infty}-\hat{\bfC}_{N})\bfB^\top)|}{\|\bfB\|_{2}^{2}}\ge 3u\right)
        \le 2d\left(\frac{29d^{2}}{s}\right)^{2s}\left(e^{-N}+2\exp\left(-\frac{1}{c}\frac{Nu^{2}}{1+u}\right)\right).
    \end{equation*}
    Setting $u=(1/36(c_{0}+2)^{2})T^{2}\exp(-2T\|\bfA_{0}\|_{\op})\le \kappa_{\min}/(18(c_{0}+2)^{2})$, we obtain
    \begin{align*}
        &\Pr\left(\inf_{\bfB\in\calC(s,c_{0})\backslash\{\bfO\}}\frac{\frac{1}{N}\sum_{i=1}^{N}\|\bfB\bx_{i}\|_{L^{2}}^{2}}{\|\bfB\|_{2}^{2}}\ge \frac{\kappa_{\min}}{2}\right)\\
        &\ge \Pr\left(\sup_{\bfB\in\calC(s,c_{0})\backslash\{\bfO\}}\frac{|\tr(\bfB(\bfC_{\infty}-\hat{\bfC}_{N})\bfB^\top)|}{\|\bfB\|_{2}^{2}}\le \frac{\kappa_{\min}}{2}\right)\\
        &\ge \Pr\left(\sup_{\bfB\in\calK(2s)\backslash\{\bfO\}}\frac{|\tr(\bfB(\bfC_{\infty}-\hat{\bfC}_{N})\bfB^\top)|}{\|\bfB\|_{2}^{2}}\le \frac{\kappa_{\min}}{6(c_{0}+2)^{2}}\right)\\
        &\ge 1-6d\left(\frac{29d^{2}}{s}\right)^{2s}\exp\left(-\frac{N}{c(c_{0}+2)^{2}}\right).
    \end{align*}
    The assumption on $N$ leads to the conclusion.
\end{proof}

\subsection{Bound on the martingale term}
The next lemma is important to control the quadratic variation of the martingale term $\bfM_{N}(T)$.
\begin{lemma}
    For some positive constant $c\ge 1$ dependent only on $T$, $\|\bfA_0\|_{\op}$, $K$, and $\|\bSigma\|_{\op}$, we have that for any $N\ge r(\bSigma)$,
    \begin{equation*}
        \Pr\left(\|\diag(\hat{\bfC}_{N})\|_{\infty}\le \kappa_{\max}+\frac{\kappa_{\min}}{2}\right)\ge 1-2d e^{-N/c}.
    \end{equation*}
\end{lemma}

\begin{proof}
    Using $N\ge r(\bSigma)$, we have that for some $c\ge 1$ dependent only on $T$, $\|\bfA_0\|_{\op}$, $K$, and $\|\bSigma\|_{\op}$, for all $u>0$,
    \begin{equation*}
        \sup_{\bv\in\bbS^{d-1}}\Pr\left(\left|\bv^\top\left(\hat{\bfC}_{N}-\bfC_{\infty}\right)\bv\right|\ge u\right)\le 6\exp\left(-\frac{Nu}{c(1+u)}\right).
    \end{equation*}
    Therefore, it holds that 
    \begin{align*}
        &\Pr\left(\|\diag(\hat{\bfC}_{N})\|_{\infty}\ge \kappa_{\max}+\frac{\kappa_{\min}}{2}\right)=\Pr\left(\max_{i=1,\ldots,d}\be_{i}^\top\hat{\bfC}_{N}\be_{i}\ge \sup_{\bv\in\bbS^{d-1}}\bv^\top\bfC_{\infty}\bv+\frac{\kappa_{\min}}{2}\right)\\
        &\le \Pr\left(\max_{i=1,\ldots,d}\be_{i}^\top\left(\hat{\bfC}_{N}-\bfC_{\infty}\right)\be_{i}\ge \frac{\kappa_{\min}}{2}\right)
        \le \Pr\left(\max_{i=1,\ldots,d}\left|\be_{i}^\top\left(\hat{\bfC}_{N}-\bfC_{\infty}\right)\be_{i}\right|\ge \frac{\kappa_{\min}}{2}\right)\\
        &\le \sum_{i=1}^{d}\Pr\left(\left|\be_{i}^\top\left(\hat{\bfC}_{N}-\bfC_{\infty}\right)\be_{i}\right|\ge \frac{\kappa_{\min}}{2}\right)
        \le 6d\exp\left(-\frac{N}{c}\right),
    \end{align*}
    where we used $\kappa_{\min}\ge (1/2)T^{2}\exp(-2T\|\bfA_{0}\|_{\op})$.
\end{proof}
The next proposition estimates the martingale term. 
\begin{proposition}[Proposition 3.5 of \citealp{ciolek2020dantzig}]\label{prop:ciolec2020dantzig:prop0305}
    There exists a positive constant $c\ge 1$ dependent only on $T,\|\bfA_{0}\|_{\op},K,\|\bSigma\|_{\op}$ such that for any $\epsilon_{0}\in(0,1)$, if $N\ge r(\bSigma)\vee (c\log(2d/\epsilon_{0}))$, it holds that
    \begin{equation*}
        \Pr\left(\sup_{\bfB\in\R^{d\times d}:\bfB\neq\bfO}\frac{\langle \bfM_{N}(T),\bfB\rangle }{\|\bfB\|_{1}}\ge \sqrt{\frac{c\log(2d^{2}/\epsilon_{0})}{N}}\right)\le \epsilon_{0}.
    \end{equation*}
\end{proposition}

\begin{proof}
    Note that $\bfM_{N}^{i,j}(T)=\frac{1}{N}\sum_{k=1}^{N}\int_{0}^{T}\diff \bw_{k}^{(i)}(t)\bx_{k}^{(j)}(t)$ is a continuous martingale whose quadratic variation is bounded above by $\|\diag(\hat{\bfC}_{N})\|_{\infty}/N$.
    We have that for some $c\ge 1$ dependent only on $T,\|\bfA_{0}\|_{\op},K,\|\bSigma\|_{\op}$, for any $u>0$,
    \begin{align*}
        &\Pr\left(\sup_{\bfB\in\R^{d\times d}:\bfB\neq\bfO}\frac{\langle \bfM_{N}(T),\bfB\rangle }{\|\bfB\|_{1}}\ge u\right)\\
        &\le \Pr\left(\sup_{\bfB\in\R^{d\times d}:\bfB\neq\bfO}\frac{\langle \bfM_{N}(T),\bfB\rangle }{\|\bfB\|_{1}}\ge u,\|\diag(\hat{\bfC}_{N})\|_{\infty}\le \kappa_{\max}+\frac{\kappa_{\min}}{2}\right)\\
        &\quad+\Pr\left(\|\diag(\hat{\bfC}_{N})\|_{\infty}\le \kappa_{\max}+\frac{\kappa_{\min}}{2}\right)\\
        &\le \sum_{i=1}^{d}\sum_{j=1}^{d}\Pr\left(\bfM_{N}^{i,j}(T)\ge u,\|\diag(\hat{\bfC}_{N})\|_{\infty}\le \kappa_{\max}+\frac{\kappa_{\min}}{2}\right)+\frac{\epsilon_{0}}{2}\\
        &\le d^{2}\exp\left(-\frac{Nu^{2}}{c}\right)+\frac{\epsilon_{0}}{2},
    \end{align*}
    where we used Bernstein's inequality for continuous local martingale $\{M_{t}\}_{t\ge0}$ with $M_{0}=0$ and its quadratic variation $\{\langle M\rangle_{t}\}_{t\ge0}$ such that
    \begin{equation*}
        \Pr\left(M_{t}\ge a,\langle M\rangle_{t}\le b\right)\le \exp\left(-a^{2}/(2b)\right).
    \end{equation*}
    Hence, we obtain the desired conclusion.
\end{proof}

\subsection{Oracle inequality}

Let $s_{0}=\|\bfA_{0}\|_{0}$. We obtain the following oracle inequality.
\begin{theorem}
    Fix $\gamma>0$ and $\epsilon_{0}\in(0,1)$. 
    There exists a positive constant $c\ge1$ dependent only on $T,\|\bfA_{0}\|_{\op},K,\|\bSigma\|_{\op}$ such that
    for $\lambda_{\Lasso}$ satisfying 
    \begin{equation*}
        \lambda_{\Lasso} \ge 2\sqrt{\frac{c\log(2d^{2}/\epsilon_{0})}{N}}
    \end{equation*}
    and $N\ge r(\bSigma)\vee [c(5+4/\gamma)^{2}(2s_{0}\log(29d^{2}/s_{0})+\log (6d)+\log(2/\epsilon_{0}))]$, with probability at least $1-\epsilon_{0}$, it holds that
    \begin{equation*}
        \frac{1}{N}\sum_{i=1}^{N}\|(\hat{\bfA}_{\Lasso}-\bfA_{0})\bx_{i}\|_{L^{2}}^{2}
        \le \left(\frac{1+\gamma}{1-\gamma}\right)\inf_{\bfA:\|\bfA\|_{0}\le s_{0}}\left\{\frac{1}{N}\sum_{i=1}^{N}\left\|(\bfA-\bfA_{0})\bx_{i}\right\|_{L^{2}}^{2}+\frac{16\lambda_{\Lasso}^{2}\|\bfA\|_{0}}{\kappa_{\min}\gamma(1+\gamma)}\right\}.
    \end{equation*}
\end{theorem}

\begin{proof}
    Owing to Propositions \ref{prop:ciolek2020dantzig:thm0303} (with $c_{0}=3+4/\gamma$) and \ref{prop:ciolec2020dantzig:prop0305}, there exists a positive constant $c\ge1$ dependent only on $T,\|\bfA_{0}\|_{\op},K,\|\bSigma\|_{\op}$ such that if $N\ge r(\bSigma)\vee [c(5+4/\gamma)^{2}(2s_{0}\log(29d^{2}/s_{0})+\log (6d)+\log(2/\epsilon_{0}))]$, all the following events hold simultaneously with probability at least $1-\epsilon_{0}$:
    \begin{align}
        &\sup_{\bfB\in\R^{d\times d}:\bfB\neq\bfO}\frac{\langle \bfM_{N}(T),\bfB\rangle }{\|\bfB\|_{1}}\le \sqrt{\frac{c\log(4d^{2}/\epsilon_{0})}{N}},\label{eq:oracle:r:martingale}\\
        &\inf_{\bfB\in\calC(s,3+4/\gamma)\backslash\{\bfO\}}\frac{\frac{1}{N}\sum_{i=1}^{N}\|\bfB\bx_{i}\|_{L^{2}}^{2}}{\|\bfB\|_{2}^{2}}\ge \frac{\kappa_{\min}}{2}.\label{eq:oracle:r:eigenvalue}
    \end{align}
    Hereafter, suppose that both Eqs.~\eqref{eq:oracle:r:martingale} and \eqref{eq:oracle:r:eigenvalue} hold.
    Lemma \ref{lem:upper:GM19} and Eq.~\eqref{eq:oracle:r:martingale} give that for any $\bfA\in\R^{d\times d}$ with $\|\bfA\|_{0}\le s_{0}$,
    \begin{align}
        &\frac{1}{N}\sum_{i=1}^{N}\left\|(\hat{\bfA}_{\Lasso}-\bfA_{0})\bx_{i}\right\|_{L^{2}}^{2}+\lambda_{\Lasso}\left\|\hat{\bfA}_{\Lasso}-\bfA\right\|_{1}\notag\\
        &\le \frac{1}{N}\sum_{i=1}^{N}\left\|(\bfA-\bfA_{0})\bx_{i}\right\|_{L^{2}}^{2}
        +2\langle \bfM_{N}(T),\bfA-\hat{\bfA}_{\Lasso}\rangle+\lambda_{\Lasso}\left\|\hat{\bfA}_{\Lasso}-\bfA\right\|_{1}
        +2\lambda_{\Lasso}\left(\left\|\bfA\right\|_{1}-\left\|\hat{\bfA}_{\Lasso}\right\|_{1}\right)\notag\\
        &\le \frac{1}{N}\sum_{i=1}^{N}\left\|(\bfA-\bfA_{0})\bx_{i}\right\|_{L^{2}}^{2}
        +\left(2\sqrt{\frac{c\log (4d^{2}/\epsilon_{0})}{N}}+\lambda_{\Lasso}\right)\left\|\hat{\bfA}_{\Lasso}-\bfA\right\|_{1}
        +2\lambda_{\Lasso}\left\|\bfA-\hat{\bfA}_{\Lasso}\right\|_{1}\notag\\
        &\le \frac{1}{N}\sum_{i=1}^{N}\left\|(\bfA-\bfA_{0})\bx_{i}\right\|_{L^{2}}^{2}+4\lambda_{\Lasso}\left\|\hat{\bfA}_{\Lasso|\calA}-\bfA\right\|_{1}\label{eq:oracle:r:basic:1}\\
        &\le \frac{1}{N}\sum_{i=1}^{N}\left\|(\bfA-\bfA_{0})\bx_{i}\right\|_{L^{2}}^{2}+4\lambda_{\Lasso}\sqrt{\|\bfA\|_{0}}\left\|\hat{\bfA}_{\Lasso|\calA}-\bfA\right\|_{2}\label{eq:oracle:r:basic:2}
    \end{align}
    where $\calA=\supp(\bfA)\subset\{1,\ldots,d\}^{2}$.
    For all $\bfA\in\R^{d\times d}$ with $\|\bfA\|_{0}\le s_{0}$ and $4\lambda_{\Lasso}\|\hat{\bfA}_{\Lasso|\calA}-\bfA\|_{1}\le (\gamma/N)\sum_{i=1}^{N}\|(\bfA-\bfA_{0})\bx_{i}\|_{L^{2}}^{2}$, it immediately holds that
    \begin{align*}
        &\frac{1}{N}\sum_{i=1}^{N}\|(\hat{\bfA}_{\Lasso}-\bfA_{0})\bx_{i}\|_{L^{2}}^{2}
        \le (1+\gamma)\left\{\frac{1}{N}\sum_{i=1}^{N}\|(\bfA-\bfA_{0})\bx_{i}\|_{L^{2}}^{2}\right\}\\
        &\le\left(\frac{1+\gamma}{1-\gamma}\right)\left\{\frac{1}{N}\sum_{i=1}^{N}\left\|(\bfA-\bfA_{0})\bx_{i}\right\|_{L^{2}}^{2}+\frac{16\lambda_{\Lasso}^{2}\|\bfA\|_{0}}{\kappa_{\min}\gamma(1+\gamma)}\right\}.
    \end{align*}
    Hence, we only consider $\bfA$ with $\|\bfA\|_{0}\le s_{0}$ with $4\lambda_{\Lasso}\|\hat{\bfA}_{\Lasso|\calA}-\bfA\|_{1}> (\gamma/N)\sum_{i=1}^{N}\|(\bfA-\bfA_{0})\bx_{i}\|_{L^{2}}^{2}$.
    For such $\bfA$, Eq.~\eqref{eq:oracle:r:basic:1} leads to the inequality
    \begin{align*}
        \lambda_{\Lasso}\left\|\hat{\bfA}_{\Lasso}-\bfA\right\|_{1}&\le \frac{1}{N}\sum_{i=1}^{N}\left\|(\hat{\bfA}_{\Lasso}-\bfA_{0})\bx_{i}\right\|_{L^{2}}^{2}+\lambda_{\Lasso}\left\|\hat{\bfA}_{\Lasso}-\bfA\right\|_{1}\le \lambda_{\Lasso}(4+4/\gamma)\|\hat{\bfA}_{\Lasso|\calA}-\bfA\|_{1}\\
        &\le \lambda_{\Lasso}(4+4/\gamma)\|\hat{\bfA}_{\Lasso}-\bfA\|_{1}.
    \end{align*}
    This implies $\hat{\bfA}_{\Lasso}-\bfA\in\calC(s_{0},3+4/\gamma)$.
    Eq.~\eqref{eq:oracle:r:eigenvalue} leads to the bound
    \begin{equation*}
        \left\|\hat{\bfA}_{\Lasso|\calA}-\bfA\right\|_{2}^{2}\le \left\|\hat{\bfA}_{\Lasso}-\bfA\right\|_{2}^{2}\le \frac{2}{\kappa_{\min}}\frac{1}{N}\sum_{i=1}^{N}\left\|(\hat{\bfA}_{\Lasso}-\bfA)\bx_{i}\right\|_{L^{2}}^{2}.
    \end{equation*}
    Therefore, Eq.~\eqref{eq:oracle:r:basic:2} yields that
    \begin{align*}
        &\frac{1}{N}\sum_{i=1}^{N}\left\|(\hat{\bfA}_{\Lasso}-\bfA_{0})\bx_{i}\right\|_{L^{2}}^{2}\le \frac{1}{N}\sum_{i=1}^{N}\left\|(\bfA-\bfA_{0})\bx_{i}\right\|_{L^{2}}^{2}+4\lambda_{\Lasso}\sqrt{\|\bfA\|_{0}}\left\|\hat{\bfA}_{\Lasso|\calA}-\bfA\right\|_{2}\\
        &\le \frac{1}{N}\sum_{i=1}^{N}\left\|(\bfA-\bfA_{0})\bx_{i}\right\|_{L^{2}}^{2}+4\lambda_{\Lasso}\sqrt{\frac{2\|\bfA\|_{0}}{\kappa_{\min}}\frac{1}{N}\sum_{i=1}^{N}\left\|(\hat{\bfA}_{\Lasso}-\bfA)\bx_{i}\right\|_{L^{2}}^{2}}\\
        &\le \frac{1}{N}\sum_{i=1}^{N}\left\|(\bfA-\bfA_{0})\bx_{i}\right\|_{L^{2}}^{2}+4\lambda_{\Lasso}\sqrt{\frac{2\|\bfA\|_{0}}{\kappa_{\min}}}\sqrt{\frac{1}{N}\sum_{i=1}^{N}\left\|(\hat{\bfA}_{\Lasso}-\bfA_{0})\bx_{i}\right\|_{L^{2}}^{2}}\\
        &\quad+4\lambda_{\Lasso}\sqrt{\frac{2\|\bfA\|_{0}}{\kappa_{\min}}}\sqrt{\frac{1}{N}\sum_{i=1}^{N}\left\|(\bfA-\bfA_{0})\bx_{i}\right\|_{L^{2}}^{2}}\\
        &\le \frac{1}{N}\sum_{i=1}^{N}\left\|(\bfA-\bfA_{0})\bx_{i}\right\|_{L^{2}}^{2}+\frac{16\lambda_{\Lasso}^{2}\|\bfA\|_{0}}{\kappa_{\min}\gamma}+\frac{\gamma}{N}\sum_{i=1}^{N}\left\|(\hat{\bfA}_{\Lasso}-\bfA_{0})\bx_{i}\right\|_{L^{2}}^{2}+\frac{\gamma}{N}\sum_{i=1}^{N}\left\|(\bfA-\bfA_{0})\bx_{i}\right\|_{L^{2}}^{2},
    \end{align*}
    where we used $2xy\le ax^{2}+a^{-1}y^{2}$ for any $x,y\in\R$ and $a>0$.
    Arranging this inequality, we obtain
    \begin{equation*}
        \frac{1}{N}\sum_{i=1}^{N}\left\|(\hat{\bfA}_{\Lasso}-\bfA_{0})\bx_{i}\right\|_{L^{2}}^{2}\le \left(\frac{1+\gamma}{1-\gamma}\right)\frac{1}{N}\sum_{i=1}^{N}\left\|(\bfA-\bfA_{0})\bx_{i}\right\|_{L^{2}}^{2}+\frac{16\lambda_{\Lasso}^{2}\|\bfA\|_{0}}{\kappa_{\min}\gamma(1-\gamma)}.
    \end{equation*}
    This leads to the desired conclusion.
\end{proof}

\subsection{Remark on rate optimality}
The derived result guarantees the validity of the Lasso estimation under $s_{0}\log(d^2)/N\ll 1$ (small sparsity of $\bfA_{0}$) and $\er(\bSigma)/N \ll1$ (small effective rank of $\bSigma$).
It successfully removes the balance condition $d\ll N$, which is required for the ergodic cases \citep[e.g.,][]{ciolek2020dantzig,dexheimer2024lasso}.
However, the convergence guarantee is slightly suboptimal since the optimal rate is $s_{0}\log(d^{2}/s_{0})/N$ as shown by Proposition \ref{prop:lower}.
We leave it as a future work.

Note that extension of Theorem 4.1 \citet{bellec2018slope} is a prospective strategy to achieve the optimal rate.
However, the extension is not straightforward as \citet{dexheimer2024lasso} point out; \citet{bellec2018slope} employ the Gaussian isoperimetry (in particular, the concentration of Lipschitz transformations around their medians) and the fact that the supremum of an empirical process is a Lipschitz transformation of a Gaussian random vector (first order Wiener chaos).
In our case, the corresponding supremum is a Lipschitz transformation of a \emph{second} order Wiener chaos rather than a first order one; hence, it would be necessary to employ or develop tools for the extension.

\section{Numerical experiment with large $N$ and fixed $d$}\label{append:numerical}
We show a numerical experiment with varying $N=400,800,1200,1600,2000$ and fixed $d=15$.
The other settings are identical to those in Section \ref{sec:numerical}.
Figure \ref{fig:append:numerical} displays the result of the numerical experiment; it exhibits the squared $\ell^2$ distances (top) and $\ell^1$ distances between the true values and MLE (left), Lasso (middle), and Slope (right) respectively.
Since $d$ and $s$ are fixed, their performances get close to each other when $N$ is large.
We can see that both Lasso and Slope improve the performance in small $N$.
\begin{figure}[ht]
    \centering
    \includegraphics[width=0.9\linewidth]{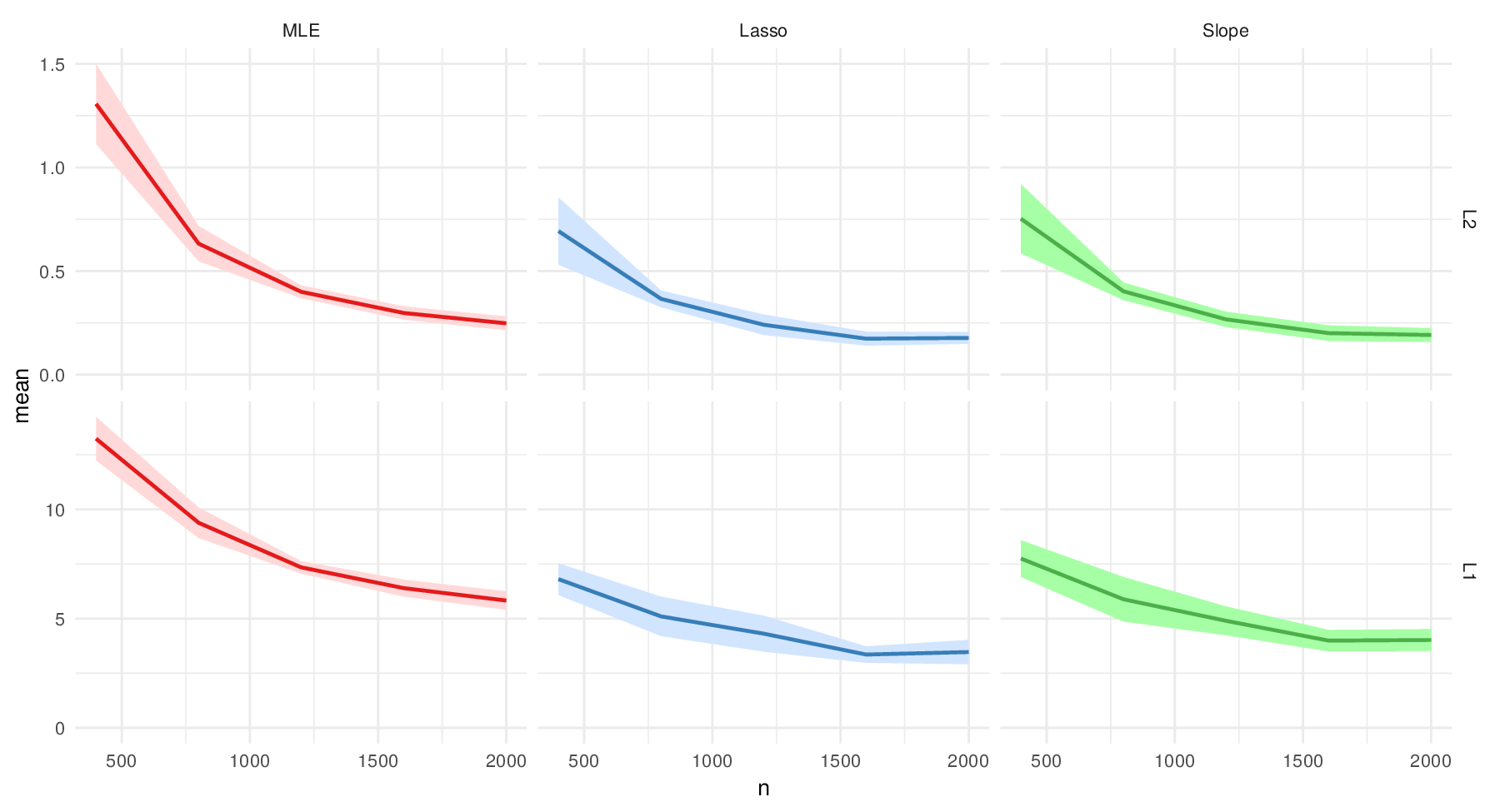}
    \caption{The comparison of squared $\ell^2$ distances (top) and $\ell^1$ distances (bottom) between the true value and the MLE (left), Lasso (middle), and Slope (right). The coloured areas represent (mean) $\pm$ (standard deviation).}
    \label{fig:append:numerical}
\end{figure}

\section*{Acknowledgement}
The author gratefully acknowledges the two anonymous reviewers for helpful comments.
This work was supported by JSPS KAKENHI Grant Number JP25K21160, JST PRESTO Grant Number JPMJPR24K7, and JST CREST Grant Number JPMJCR2115.
\subsection*{Disclosure of AI use}
The author used ChatGPT 4.0 and Claude (Opus 5.5) to improve the English language and draft portions of the simulation code scaffolding and performed all verification and revisions.

\bibliographystyle{apalike}
\bibliography{bibabbr}

\end{document}

%% file: preamble.tex
\newcommand{\DateOfPub}[1]{}

\usepackage[UKenglish]{babel}

\usepackage{amsfonts,amsmath,amsthm,amssymb,bbm} 
\usepackage{mathrsfs} 
\usepackage{stmaryrd} 
\usepackage{mathtools}

\usepackage{lscape} 

\usepackage{natbib}
\usepackage{url}

\usepackage{tikz-cd}
\usepackage{ascmac}
\usepackage{soul}

\usepackage{comment} 
\usepackage{graphicx} 
\usepackage{setspace}
\usepackage{color}

\usepackage{multirow}

\allowdisplaybreaks[1]

\newtheorem{theorem}{Theorem}[section]
\newtheorem{proposition}[theorem]{Proposition}
\newtheorem{lemma}[theorem]{Lemma}
\newtheorem{corollary}[theorem]{Corollary}

\theoremstyle{remark}
\newtheorem{remark}{Remark}

\theoremstyle{definition}
\newtheorem{assumption}{Assumption}

\newcommand\EatDot[1]{}

\newcommand{\bX}{\boldsymbol{X}}

\newcommand{\bb}{\boldsymbol{b}}

\newcommand{\be}{\boldsymbol{e}}

\newcommand{\br}{\boldsymbol{r}}

\newcommand{\bu}{\boldsymbol{u}}
\newcommand{\bv}{\boldsymbol{v}}
\newcommand{\bw}{\boldsymbol{w}}
\newcommand{\bx}{\boldsymbol{x}}
\newcommand{\by}{\boldsymbol{y}}
\newcommand{\bz}{\boldsymbol{z}}

\newcommand{\bxi}{\boldsymbol{\xi}}

\newcommand{\bfA}{\mathbf{A}}
\newcommand{\bfB}{\mathbf{B}}
\newcommand{\bfC}{\mathbf{C}}
\newcommand{\bfD}{\mathbf{D}}

\newcommand{\bfG}{\mathbf{G}}

\newcommand{\bfI}{\mathbf{I}}

\newcommand{\bfK}{\mathbf{K}}

\newcommand{\bfM}{\mathbf{M}}

\newcommand{\bfO}{\mathbf{O}}

\newcommand{\bfU}{\mathbf{U}}
\newcommand{\bfV}{\mathbf{V}}

\newcommand{\bSigma}{\mathbf{\Sigma}}

\newcommand{\bbD}{\mathbb{D}}

\newcommand{\bbH}{\mathbb{H}}

\newcommand{\bbP}{\mathbb{P}}

\newcommand{\bbS}{\mathbb{S}}

\newcommand{\calA}{\mathcal{A}}

\newcommand{\calC}{\mathcal{C}}
\newcommand{\calD}{\mathcal{D}}

\newcommand{\calI}{\mathcal{I}}

\newcommand{\calK}{\mathcal{K}}
\newcommand{\calL}{\mathcal{L}}

\newcommand{\calN}{\mathcal{N}}
\newcommand{\calO}{\mathcal{O}}

\newcommand{\calS}{\mathcal{S}}

\newcommand{\N}{\mathbb{N}}

\newcommand{\R}{\mathbb{R}}

\newcommand{\E}{\mathbb{E}}
\renewcommand{\Pr}{\mathbb{P}}

\newcommand{\tr}{\mathrm{tr}}
\newcommand{\diag}{\mathrm{diag}}

\newcommand{\diff}{\mathrm{d}}

\newcommand{\zero}{\mathbf{0}}
\DeclareMathOperator*{\argmin}{arg\,min}

\DeclareMathOperator*{\minimize}{minimize}

\newcommand{\er}{\mathbf{r}}

\newcommand{\Loss}{\mathcal{L}}
\newcommand{\Reg}{\textnormal{Reg}}
\newcommand{\HS}{\mathbb{H}}

\newcommand{\supp}{\textnormal{supp}}
\newcommand{\Lasso}{\textnormal{L}}
\newcommand{\Slope}{\textnormal{S}}

\newcommand{\op}{\textnormal{op}}
\newcommand{\rad}{\textnormal{rad}}
\newcommand{\vect}{\textnormal{vec}}
\newcommand{\ind}{\mathbbm{1}}